\documentclass[11pt]{article}
\usepackage{amsfonts, amsmath, amssymb, amscd, amsthm, color, graphicx, mathrsfs, mathabx, wasysym, setspace, mdwlist, calc,float}
\usepackage{setspace}
\usepackage{hyperref}
\usepackage{tikz-cd} 
\usepackage{hyperref}
\usepackage{tocloft}
\usepackage{xcolor}

\usepackage{hyperref}
\hypersetup{linktocpage}

\hypersetup{colorlinks,
    linkcolor={red!50!black},
    citecolor={blue!80!black},
    urlcolor={blue!80!black}}
\usepackage{float}

\newcommand{\e}{\varepsilon}
\newcommand{\NN}{\mathbb{N}}

\newcommand{\ZZ}{\mathbb{Z}}
\newcommand{\RR}{\mathbb{R}}

\newtheorem{thm}{Theorem}[section]
\newtheorem{cor}[thm]{Corollary}
\newtheorem{lem}[thm]{Lemma}
\newtheorem{prop}[thm]{Proposition}

\theoremstyle{definition}
\newtheorem{defn}[thm]{Definition}

\theoremstyle{remark}
\newtheorem{rem}[thm]{Remark}

\newcommand{\tD}{\widetilde{D}}

\renewcommand{\P}{\mathcal{P}}

\newcommand{\lelex}{\le_{\rm lex}}
\newcommand{\stab}{\mathrm{Stab}}

\renewcommand{\-}{\mathchar`-}

\renewcommand{\H}{\mathcal{H}}
\newcommand{\G}{\mathcal{G}}
\newcommand{\hd}{\widehat{d}}
\newcommand{\act}{\curvearrowright}
\newcommand{\inj}{\hookrightarrow}

\newcommand\blfootnote[1]{%
  \begingroup
  \renewcommand\thefootnote{}\footnote{#1}%
  \addtocounter{footnote}{-1}%
  \endgroup
}

\begin{document}

\title{Hyperfiniteness of boundary actions of acylindrically hyperbolic groups}

\author{Koichi Oyakawa}
\date{Vanderbilt University}

\maketitle

\vspace{-3mm}

\begin{abstract}
We prove that for any countable acylindrically hyperbolic group $G$, there exists a generating set $S$ of $G$ such that the corresponding Cayley graph $\Gamma(G,S)$ is hyperbolic, $|\partial\Gamma(G,S)|>2$, the natural action of $G$ on $\Gamma(G,S)$ is acylindrical, and the natural action of $G$ on the Gromov boundary $\partial\Gamma(G,S)$ is hyperfinite. This result broadens the class of groups that admit a non-elementary acylindrical action on a hyperbolic space with a hyperfinite boundary action.
\end{abstract}

\section{Introduction}
\blfootnote{\textbf{MSC} Primary: 20F65. Secondary: 03E15, 37A20.}
\blfootnote{\textbf{Key words and phrases}: acylindrically hyperbolic groups, hyperfiniteness, Borel equivalence relations.}
Hyperfiniteness is a property of countable Borel equivalence relations that measures their complexity. It is a classical topic in descriptive set theory that has been attracting people's interest for decades and its research is still active to this day. Because any countable Borel equivalence relation is the orbit equivalence relation of a Borel action of a countable group by the Feldman-Moore theorem (see Definition \ref{def:orbit equivalence} and \cite{FM}), people have investigated orbit equivalence relations of group actions. Historically, study on amenable groups preceded toward a long standing open problem asking whether all orbit equivalence relations of Borel actions of countable amenable groups are hyperfinite. Remarkable progress on this problem includes partial yet crucial results for $\ZZ^n$ in \cite{Weiss}, finitely generated groups with polynomial growth in \cite{JKL02}, countable Abelian groups in \cite{GJ}, and polycyclic groups in \cite{CJM+}.

On the other hand, there was not much progress made for non-amenable groups until very recently and this is the focus of this paper. To the best of my knowledge, the only result of hyperfiniteness in absence of measures in non-amenable case before 2010s was obtained by Dougherty, Jackson and Kechris in \cite{DJK94}, where they proved that the action of the free group $F_2$ on the Gromov boundary is hyperfinite. A breakthrough in this direction was achieved by Huang, Sabok, and Shinko in \cite{HSS} where they generalized this result to cubulated hyperbolic groups. In \cite{MS}, Marquis and Sabok succeeded in proving hyperfiniteness of boundary actions of hyperbolic groups in full generality. This was further generalized to finitely generated relatively hyperbolic groups by Karpinski in \cite{Kar}. Other important results in non-amenable case are \cite{PS} by Przytycki and Sabok, where they proved that actions of mapping class groups on the Gromov boundaries of the arc complex and the curve complex are hyperfinite, and \cite{NV} by Naryshkin and Vaccaro, where they proved that boundary actions of hyperbolic groups have finite Borel asymptotic dimension, which strengthen \cite{MS}.

Most of the above results in non-amenable case can be summarized by saying that the involved groups admit a non-elementary acylindrical action on a hyperbolic space with a hyperfinite action on the Gromov boundary. In this paper, we show that this is true for a much wider class of groups by proving the following theorem.
\begin{thm} \label{thm:main}
    For any countable acylindrically hyperbolic group $G$, there exists a generating set $S$ of $G$ such that the corresponding Cayley graph $\Gamma(G,S)$ is hyperbolic, $|\partial\Gamma(G,S)|>2$, the natural action of $G$ on $\Gamma(G,S)$ is acylindrical, and the natural action of $G$ on the Gromov boundary $\partial\Gamma(G,S)$ is hyperfinite.
\end{thm}
Note that acylindrically hyperbolic groups do not need to be finitely generated. The new portion of Theorem \ref{thm:main} is the hyperfiniteness of the action on the Gromov boundary, while the other conditions were proved by Osin in \cite{Osin16}. Also, the generating set $S$ in Theorem \ref{thm:main} is the same as the one constructed in \cite[Theorem 5.4]{Osin16}. The class of acylindrically hyperbolic groups is broad and includes many examples of interest: non-elementary hyperbolic and relatively hyperbolic groups, all but finitely many mapping class groups of punctured closed surfaces, $Out(F_n)$ for any $n\ge 2$, directly indecomposable right angled Artin groups, non virtually cyclic graphical small cancellation groups including some Gromov monsters (see \cite{GS}), one relator groups with at least 3 generators, Higman group, most orientable 3-manifold groups (see \cite{MO}), and many others. Proving hyperfiniteness in this broad class faces some difficulties that didn't appear in previous results. For exmaple, for acylindrically hyperbolic groups, there is no canonical generating set in general, local compactness of geodesic ray bundles is lacking (see \cite[Section 1]{MS}), and elements of the Gromov boundary may not be represented by geodesic rays. We circumvent these difficulties by bringing a new insight on the Gromov boundaries of acylindrically hyperbolic groups, which we explain in Section \ref{sec:Gromov} and Section \ref{sec:appendix}, and by building on the work of Naryshkin and Vaccaro in \cite{NV}. In \cite{NV}, given a hyperbolic group $G$ with a finite symmetric generating set $S$, they constructed an injective Borel measurable map from $\partial G$ to $S^\NN$ that Borel reduces a finite index subequivalence relation of the orbit equivalence relation $E_G^{\partial G}$ to the tail equivalence relation $E_t(S)$, thereby hyperfiniteness of $E_t(S)$ implied hyperfiniteness of $E_G^{\partial G}$.

Moreover, Theorem \ref{thm:main} has the following application to topological amenability of group actions. Corollary \ref{cor:main} is interesting, because it contrasts with the fact that some Gromov monsters are acylindrically hyperbolic and these groups don't admit a topologically amenable action on any compact Hausdorff space as they're non-exact.

\begin{cor}\label{cor:main}
For any countable acylindrically hyperbolic group $G$, there exists a generating set $S$ of $G$ such that the corresponding Cayley graph $\Gamma(G,S)$ is hyperbolic, $|\partial\Gamma(G,S)|>2$, the natural action of $G$ on $\Gamma(G,S)$ is acylindrical, and the natural action of $G$ on the Gromov boundary $\partial\Gamma(G,S)$ is topologically amenable.
\end{cor}

The paper is organized as follows. In Section \ref{sec:prelim}, we discuss the necessary definitions and known results about hyperfinite Borel equivalence relations and acylindrically hyperbolic groups. In Section \ref{sec:Gromov}, we introduce a new way to represent elements of the Gromov boundary of an acylindrically hyperbolic group for a nice generating set by building on the work of Osin in \cite{Osin16}. In Section \ref{sec:main thm}, we give a proof of Theorem \ref{thm:main} by using techniques developed in Section \ref{sec:Gromov}. In Section \ref{sec:application to topologically amenable actions}, we introduce topological amenability of group actions and prove Corollary \ref{cor:main}. In Section \ref{sec:appendix}, we summarize more results on the Gromov boundaries of acylindrically hyperbolic groups that are not necessary in the proof of Theorem \ref{thm:main} but are of independent interest for possible future use. Section \ref{sec:appendix} is independent of Section \ref{sec:main thm} and Section \ref{sec:application to topologically amenable actions} and can be read as another continuation of Section \ref{sec:Gromov}.

\noindent\textbf{Acknowledgment.}
I would like to thank the anonymous referee for many helpful comments, which greatly improved exposition of this paper.

\section{Preliminary} \label{sec:prelim}

\subsection{Descriptive set theory} \label{subsec:descriptiveset theory}

In this section, we review concepts in descriptive set theory.

\begin{defn}\label{def:Polish}
    A \emph{Polish space} is a separable completely metrizable topological space.
\end{defn}

\begin{defn}
    A measurable space $(X,\mathcal{B})$ is called a \emph{standard Borel space}, if there exists a topology $\mathcal{O}$ on $X$ such that $(X,\mathcal{O})$ is a Polish space and $\mathcal{B}(\mathcal{O})=\mathcal{B}$ holds, where $\mathcal{B}(\mathcal{O})$ is the $\sigma$-algebra on $X$ generated by $\mathcal{O}$.
\end{defn}

\begin{defn}
    Let $X$ be a standard Borel space and $E$ be an equivalence relation on $X$. $E$ is called \emph{Borel}, if $E$ is a Borel subset of $X\times X$. $E$ is called \emph{countable} (resp. \emph{finite}), if for any $x\in X$, the set $\{ y\in X \mid (x,y)\in E \}$ is countable (resp. finite).
\end{defn}

\begin{rem}
    The word ``countable Borel equivalence relation" is often abbreviated to ``CBER".
\end{rem}

\begin{defn}
    Let $X$ be a standard Borel space. A countable Borel equivalence relation $E$ on $X$ is called \emph{hyperfinite}, if there exist finite Borel equivalence relations $(E_n)_{n=1}^\infty$ on $X$ such that $E_n\subset E_{n+1}$ for any $n\in\NN$ and $E=\bigcup_{n=1}^\infty E_n$.
\end{defn}

Definition \ref{def:orbit equivalence} and Definition \ref{def:tail equivalence} are two important examples of CBERs in this paper.

\begin{defn}\label{def:orbit equivalence}
     Suppose that a group $G$ acts on a set $S$. The equivalence relation $E_G^S$ on $S$ is defined as follows: for $x,y\in S$,
     \[
        (x,y) \in E_G^S \iff \exists\ g\in G {\rm ~s.t.~} gx=y.
     \]
     $E_G^S$ is called the \emph{orbit equivalence relation} on $S$.
\end{defn}

Lemma \ref{lem:orbit equivalence is cber} is straightforward, but we record the proof for convenience of readers.

\begin{lem}\label{lem:orbit equivalence is cber}
    Suppose that a countable group $G$ acts on a standard Borel space $S$ as Borel isomorphism, then $E_G^S$ is a CBER.
\end{lem}

\begin{proof}
    For any $g\in G$, the set $\mathrm{Graph}(g)$ defined by $\mathrm{Graph}(g)=\{(x,gx)\in S\times S \mid x\in S\}$ is Borel since $g\colon S \to S$ is Borel measurable. Since $G$ is countable and we have $E_G^S=\bigcup_{g\in G} \mathrm{Graph}(g)$, the set $E_G^S$ is also Borel, being the countable union of Borel sets. Finally, for any $x\in S$, the orbit equivalence class of $x$ is exactly $G x$, which is countable since $G$ is countable. Thus, $E_G^S$ is a CBER.
\end{proof}

Recall that any countable set $\Omega$ with the discrete topology is a Polish space. Hence, $\Omega^\NN$ endowed with the product topology is a Polish space.

\begin{defn}\label{def:tail equivalence}
        Let $\Omega$ be a countable set. The equivalence relation $E_t(\Omega)$ on $\Omega^\NN$ is defined as follows: for $w_0=(s_1,s_2,\cdots), w_1=(t_1,t_2,\cdots) \in \Omega^\NN$,
    \[
    (w_0,w_1)\in E_t(\Omega) \iff \exists n, \exists m \in\NN\cup\{0\} {\rm ~s.t.~}\forall i\in\NN, s_{n+i}=t_{m+i}.
    \]
    $E_t(\Omega)$ is called the \emph{tail equivalence relation} on $\Omega^\NN$.
\end{defn}

We list some facts needed for the proof of Theorem \ref{thm:main}. Proposition \ref{prop:DJK} is a particular case of \cite[Corollary 8.2]{DJK94}.

\begin{prop}\label{prop:DJK}{\rm (cf. \cite[Corollary 8.2]{DJK94})}
For any countable set $\Omega$, the tail equivalence relation $E_t(\Omega)$ on $\Omega^\NN$ is a hyperfinite CBER.
\end{prop}

\begin{prop}\label{prop:JKL}{\rm \cite[Proposition 1.3.(vii)]{JKL02}}
    Let $X$ be a standard Borel space and $E,F$ be countable Borel equivalence relations on $X$. If $E\subset F$, $E$ is hyperfinite, and every $F$-equivalence class contains only finitely many $E$-classes, then $F$ is hyperfinite.
\end{prop}

\vspace{2mm}

\subsection{The Gromov boundary of a hyperbolic space}

In this section, we review the Gromov boundary of a hyperbolic space. For more on the Gromov boundary, readers are referred to \cite{BH}.

\begin{defn}
    Let $(S,d_S)$ be a metric space. For $x,y,z\in S$, we define $(x,y)_z^S$ by
\begin{align}\label{eq:gromov product}
    (x,y)_z^S=\frac{1}{2}\left( d_S(x,z)+d_S(y,z)-d_S(x,y) \right).    
\end{align}
\end{defn}

\begin{prop} \label{prop:hyp sp}
    For any geodesic metric space $(S,d_S)$, the following conditions are equivalent.
    \item[(1)]
    There exists $\delta\in\NN$ satisfying the following property. Let $x,y,z\in S$, and let $p$ be a geodesic path from $z$ to $x$ and $q$ be a geodesic path from $z$ to $y$. If two points $a\in p$ and $b\in q$ satisfy $d_S(z,a)=d_S(z,b)\le (x,y)_z^S$, then we have $d_S(a,b) \le \delta$.
    \item[(2)] 
    There exists $\delta\in\NN$ such that for any $w,x,y,z \in S$, we have
    \[
    (x,z)_w^S \ge \min\{(x,y)_w^S,(y,z)_w^S\} - \delta.
    \]
\end{prop}

\begin{defn}
    A geodesic metric space $S$ is called \emph{hyperbolic}, if $S$ satisfies the equivalent conditions (1) and (2) in Proposition \ref{prop:hyp sp}. We call a hyperbolic space $\delta$-\emph{hyperbolic} with $\delta \in \NN$, if $\delta$ satisfies both of (1) and (2) in Proposition \ref{prop:hyp sp}. A connected graph $\Gamma$ is called \emph{hyperbolic}, if the geodesic metric space $(\Gamma,d_\Gamma)$ is hyperbolic, where $d_\Gamma$ is the graph metric of $\Gamma$.
\end{defn}

In the remainder of this section, suppose that $(S,d_S)$ is a hyperbolic geodesic metric space.

\begin{defn}\label{def:seq to infty}
    A sequence $(x_n)_{n=1}^\infty$ of elements of $S$ is said to \emph{converge to infinity}, if we have $\lim_{i,j\to\infty} (x_i,x_j)_o^S =\infty$ for some (equivalently any) $o\in S$. For two sequences $(x_n)_{n=1}^\infty,(y_n)_{n=1}^\infty$ in $S$ converging to infinity, we define the relation $\sim$ by $(x_n)_{n=1}^\infty \sim (y_n)_{n=1}^\infty$ if we have $\lim_{i,j\to\infty} (x_i,y_j)_o^S =\infty$ for some (equivalently any) $o\in S$.
\end{defn}

\begin{rem}
    It's not difficult to see that the relation $\sim$ in Definition \ref{def:seq to infty} is an equivalence relation by using the condition (2) of Proposition \ref{prop:hyp sp}.
\end{rem}

\begin{defn}
    The quotient set $\partial S$ is defined by
    \[
    \partial S = \{ {\rm sequences~in~ }S {\rm ~converging~to~infinity} \} / \sim
    \]
    and called \emph{Gromov boundary} of $S$.
\end{defn}

\begin{rem}
    The set $\partial S$ is sometimes called the sequential boundary of $S$. Note that $\partial S$ sometimes coincides with the geodesic boundary of $S$ (e.g. when $S$ is a proper metric space), but this is not the case in general.
\end{rem}

\begin{defn}
    For $o\in S$ and $\xi,\eta \in S\cup \partial S$, we define $(\xi,\eta)_o^S$ by
    \begin{equation}\label{eq:gromov prod for boundary}
        (\xi,\eta)_o^S=\sup\{ \liminf_{i,j\to\infty}(x_i,y_j)_o^S \mid \xi=[(x_n)_{n=1}^\infty], \eta=[(y_n)_{n=1}^\infty] \},
    \end{equation}
    where we define $\xi=[(x_n)_{n=1}^\infty]$ as follows. If $\xi \in \partial S$, then $(x_n)_{n=1}^\infty$ is a sequence in $S$ converging to infinity such that $\xi$ represents the equivalence class of $(x_n)_{n=1}^\infty$. If $\xi \in S$, then $(x_n)_{n=1}^\infty$ is constant with $x_n \equiv \xi$. We define $\eta=[(y_n)_{n=1}^\infty]$ in the same way.
\end{defn}

\begin{prop} \label{prop:gromov topo}
    For any hyperbolic geodesic metric space $(S,d_S)$, there exists a unique topology $\mathcal{O}_S$ on $S\cup\partial S$ such that the relative topology of $\mathcal{O}_S$ on $S$ coincides with the metric topology of $d_S$ and for any $\xi \in \partial S$ and $o\in S$, the sets $(U(o,\xi,n))_{n=1}^\infty$ defined by
    \[
    U(o,\xi,n)=\{ \eta \in S\cup\partial S \mid (\eta,\xi)_o^S > n \}
    \]
    form a neighborhood basis of $\mathcal{O}_S$ at $\xi$.
\end{prop}

\begin{rem}
    When a group $G$ acts on $S$ isometrically, this action naturally extends to the homeomorphic action on $S\cup\partial S$.
\end{rem}

The following proposition is a variation of \cite[Proposition 3.21]{BH} and can be proved in the same way. Indeed, in the statement of \cite[Proposition 3.21]{BH}, the domain of $D$ below is $(\partial S)^2$.

\begin{prop}\label{prop:visual metric}
    For any $o\in S$, there exist a map $D \colon (S\cup\partial S)^2 \to [0,\infty)$ and constants $\e,\e'>0$ with $\e'\le \sqrt{2}-1$ satisfying the following three conditions.
    \begin{itemize}
        \item[(i)]
        $D(x,y)=D(y,x)$ for any $x,y \in S\cup\partial S$.
        \item[(ii)]
        $D(x,z) \le D(x,y)+D(y,z)$ for any $x,y,z \in S\cup\partial S$. 
        \item[(iii)]
        $(1-2\e')e^{-\e (x,y)_o^S} \le D(x,y) \le e^{-\e (x,y)_o^S}$ for any $x,y \in S\cup\partial S$.
    \end{itemize}
    For convenience, if $(x,y)_o^S=\infty$, then we define $e^{-\e (x,y)_o^S}=0$.
\end{prop}

\begin{rem}\label{rem:D not metric}
    For any $x\in S$, we have $(1-2\e')e^{-\e d_S(o,x)} \le \inf_{y\in S\cup\partial S}D(x,y)$ by $\sup_{y \in S\cup\partial S} (x,y)_o^S \le d_S(o,x)$. Hence, the map $D$ in Proposition \ref{prop:visual metric} is not a metric on $(S\cup\partial S)^2$. However, $D$ is a metric on $(\partial S)^2$. This metric $D|_{(\partial S)^2}$ is called a \emph{visual metric} and the metric topology of $D|_{(\partial S)^2}$ on $(\partial S)^2$ coincides with the relative topology of $\mathcal{O}_S$ in Proposition \ref{prop:gromov topo}.
\end{rem}

\subsection{Hull-Osin's separating cosets of hyperbolically embedded subgroups} \label{subsec:HE subgroups}

In this section, we review hyperbolically embedded subgroups and Hull-Osin's separating cosets. The notion of separating cosets of hyperbolically embedded subgroups was first introduced by Hull and Osin in \cite{HO13} and further developed by Osin in \cite{Osin16}. There are two differences in the definition of separating cosets in \cite{HO13} and in \cite{Osin16}, though other terminologies and related propositions are mostly the same between them. This difference is explained in Remark \ref{difference}. With regards to this difference, we follow definitions and notations of \cite{Osin16} in our discussion. We begin with defining auxiliary concepts.

\begin{defn}
    Let $m,n\in\ZZ$ and let $\Gamma$ be a connected graph with the graph metric $d_\Gamma$. A \emph{path} $p$ in $\Gamma$ is a graph homomorphism from one of $[m,n]$, $[m,\infty)$, or $\RR$ to $\Gamma$, where each domain is considered as a graph with a vertex set $\ZZ\cap[m,n]$, $\ZZ\cap[m,\infty)$, and $\ZZ$ respectively. When we want to emphasize that the domain is $[m,\infty)$ (resp. $\RR$), we call $p$ an \emph{infinite path} (resp. \emph{bi-infinite path}). A \emph{subpath} $q$ of a path $p$ is a path obtained by restricting $p$ to a subset of the domain of $p$. For vertices $x$ and $y$ of $\Gamma$, a path $p$ \emph{from} $x$ \emph{to} $y$ is a path $p$ with the domain $[m,n]$ satisfying $p(m)=x$ and $p(n)=y$. We also denote the initial point $x$ of $p$ by $p_-$ and the terminal point $y$ by $p_+$. A path $p$ is called \emph{closed}, if $p_-=p_+$. Similarly, an infinite path $p$ \emph{from} $x$ is a path satisfying $p(m)=x$ and we denote the initial point $x$ of $p$ by $p_-$. A path $p$ is called \emph{geodesic}, if $p$ is a distance-preserving map from its domain to $(\Gamma,d_\Gamma)$. An infinite geodesic path is also called a \emph{geodesic ray}. Since geodesic paths are injective, we often identify their images with maps. For example, we will denote a geodesic ray $p$ by $p=(x_0,x_1,x_2,\cdots)$ where each $x_i$ is a vertex of $\Gamma$ and each pair $(x_i,x_{i+1})$ is adjacent. Also, for a geodesic path $p$, if $q$ is its subpath from $x$ to $y$ (resp. its infinite subpath from $x$), we denote $q$ by $p_{[x,y]}$ (resp. $p_{[x,\infty)}$).
\end{defn}

\begin{rem}
    For two paths $p$ and $q$ satisfying $p_+=q_-$, we can define the path $pq$ by concatenating $p$ and $q$. Also, for a path $p$ from $x$ to $y$, we can define the path $p^{-1}$ to be the path from $y$ to $x$ obtained by reversing the direction of $p$.
\end{rem}

\begin{rem}
    For vertices $x$ and $y$ of $\Gamma$, we sometimes denote a geodesic path from $x$ to $y$ by $[x,y]$, though it's not necessarily unique.
\end{rem}

\begin{defn}
Suppose that $G$ is a group, $X$ is a subset of $G$, and $\{H_\lambda\}_{\lambda\in\Lambda}$ is a collection of subgroups of $G$ such that the set $X\cup \bigcup_{\lambda\in\Lambda}H_\lambda$ generates $G$. We denote $\H = \bigsqcup_{\lambda \in \Lambda}(H_\lambda \setminus \{1\})$ and $X\sqcup\H$ to mean sets of labels. Note that these unions are disjoint as sets of labels, not as subsets of $G$. Let $\Gamma(G, X\sqcup\mathcal{H})$ be the Cayley graph of $G$ with respect to $X\sqcup\mathcal{H}$, which allows loops and multiple edges, that is, its vertex set is $G$ and its positive edge set is $G\times(X\sqcup\mathcal{H})$. The graph $\Gamma(G,X\sqcup\H)$ is called a \emph{relative Cayley graph}. For each $\lambda\in\Lambda$, we consider the Cayley graph $\Gamma(H_\lambda,H_\lambda\setminus\{1\})$, which is a subgraph of $\Gamma(G,X\sqcup\H)$, and define a metric $\hd_{\lambda}\colon H_\lambda \times H_\lambda \to [0,\infty]$ as follows. We say that a path $p$ in $\Gamma(G,X\sqcup\H)$ is $\lambda$-\emph{admissible}, if $p$ doesn't contain any edge of $\Gamma(H_\lambda,H_\lambda\setminus\{1\})$. Note that $p$ can contain an edge whose label is an element of $H_\lambda$ (e.g. the case when the initial vertex of the edge is not in $H_\lambda$) and also $p$ can pass vertices of $\Gamma(H_\lambda,H_\lambda\setminus\{1\})$. For $f,g\in H_\lambda$, we define $\hd_{\lambda}(f,g)$ to be the minimum of lengths of all $\lambda$-admissible paths from $f$ to $g$. If there is no $\lambda$-admissible path from $f$ to $g$, then we define $\hd_{\lambda}(f,g)$ by $\hd_{\lambda}(f,g)=\infty$. For convenience, we extend $\hd_{\lambda}$ to $\hd_{\lambda}\colon  G\times G \to [0,\infty]$ by defining $\hd_{\lambda}(f,g)=\hd_{\lambda}(1,f^{-1}g)$ if $f^{-1}g\in H_\lambda$ and $\hd_{\lambda}(f,g)=\infty$ otherwise. The metric $\hd_{\lambda}$ is called the \emph{relative metric}.
\end{defn}

\begin{defn}\label{def:hyp emb}
Suppose that $G$ is a group and $\{H_\lambda\}_{\lambda\in\Lambda}$ is a collection of subgroups of $G$. For a subset $X$ of $G$, $\{H_\lambda\}_{\lambda\in\Lambda}$ is said to be \emph{hyperbolically embedded in} $(G,X)$ (and denoted by $\{H_\lambda\}_{\lambda\in\Lambda} \inj_h (G,X)$), if it satisfies the two conditions below.
\begin{itemize}
    \item [(1)] The set $X\cup (\bigcup_{\lambda\in\Lambda}H_\lambda)$ generates $G$ and the graph $\Gamma(G,X\sqcup\H)$ is hyperbolic.
    \item [(2)] For any $\lambda\in\Lambda$, $(H_\lambda,\hd_{\lambda})$ is locally finite, i.e. for any $n\in\NN$, 
    $\{g\in H_\lambda \mid \hd_{\lambda}(1,g)\le n\}$ is finite.
\end{itemize}
A collection of subgroups $\{H_\lambda\}_{\lambda\in\Lambda}$ is also said to be \emph{hyperbolically embedded in} $G$ (and denoted by $\{H_\lambda\}_{\lambda\in\Lambda} \inj_h G$), if there exists a subset $X$ of $G$ such that $\{H_\lambda\}_{\lambda\in\Lambda} \inj_h (G,X)$.
\end{defn}

In the remainder of this section, suppose that $\{H_\lambda\}_{\lambda\in\Lambda}$ is hyperbolically embedded in $(G,X)$ as in Definition \ref{def:hyp emb}. We next prepare concepts to define separating cosets.

\begin{defn}\cite[Definition 4.1]{Osin16}
Suppose that $p$ is a path in the relative Cayley graph $\Gamma(G,X\sqcup\H)$. A subpath $q$ of $p$ is called an \emph{$H_\lambda$-subpath} if the labels of all edges of $q$ are in $H_\lambda$. In the case that $p$ is a closed path, $q$ can be a subpath of any cyclic shift of $p$. An $H_\lambda$-subpath $q$ of a path $p$ is called \emph{$H_\lambda$-component} if $q$ is not contained in any longer $H_\lambda$-subpath of $p$. In the case that $p$ is a closed path, we require that $q$ is not contained in any longer $H_\lambda$-subpath of any cyclic shift of $p$. Furthermore, by a \emph{component}, we mean an $H_\lambda$-component for some $H_\lambda$. Two $H_\lambda$-components $q_1$ and $q_2$ of a path $p$ are called \emph{connected}, if all vertices of $q_1$ and $q_2$ are in the same $H_\lambda$-coset. An $H_\lambda$-component $q$ of a path $p$ is called \emph{isolated}, if $q$ is not connected to any other $H_\lambda$-component of $p$.
\end{defn}

\begin{rem}
Note that all vertices of an $H_\lambda$-component lie in the same $H_\lambda$-coset.
\end{rem}

Proposition \ref{prop:C} is a particular case of \cite[Proposition 4.13]{DGO} and plays a crucial role in this paper.

\begin{prop}\emph{\cite[Lemma 2.4]{HO13}}\label{prop:C}
There exists a constant $C>0$ such that for any geodesic $n$-gon $p$ in $\Gamma(G,X\sqcup\H)$ and any isolated $H_\lambda$-component $a$ of $p$, we have
\[\hd_{\lambda}(a_-,a_+)\le nC.\]
\end{prop}

In the remainder of section, we fix any constant $D>0$ with
\begin{equation}\label{eq:D>3C}
D \ge 3C.
\end{equation}

We can now define separating cosets.

\begin{defn}\cite[Definition 4.3]{Osin16}
\label{separating cosets}
A path $p$ in $\Gamma(G,X\sqcup\H)$ is said to \emph{penetrate} a coset $xH_\lambda$ for some $\lambda\in\Lambda$, if $p$ decomposes into $p_1ap_2$, where $p_1,p_2$ are possibly trivial, $a$ is an $H_\lambda$-component, and $a_-\in xH_\lambda$. Note that if $p$ is geodesic, $p$ penetrates any coset of $H_\lambda$ at most once. In this case, $a$ is called the \emph{component of} $p$ \emph{corresponding to} $xH_\lambda$ and also the vertices $a_-$ and $a_+$ are called the \emph{entrance and exit points} of $p$ and are denoted by $p_{in}(xH_\lambda)$ and $p_{out}(xH_\lambda)$ respectively. If in addition we have $\hd_{\lambda}(a_-,a_+)>D$, then $p$ is said to \emph{essentially penetrates} $xH_\lambda$. For $f,g\in G$ and $\lambda\in\Lambda$, if there exists a geodesic path from $f$ to $g$ in $\Gamma(G,X\sqcup\H)$ which essentially penetrates an $H_\lambda$-coset $xH_\lambda$, then $xH_\lambda$ is called an $(f,g;D)$-\emph{separating coset}. The set of all $(f,g;D)$-separating cosets of subgroups from the collection $\{H_\lambda\}_{\lambda\in\Lambda}$ is denoted by $S(f,g;D)$.
\end{defn}

\begin{rem}\label{rem:component of geodesic}
    If a geodesic path $p$ penetrates an $H_\lambda$-coset $xH_\lambda$, then the component $a$ of $p$ corresponding to $xH_\lambda$ consists of a single edge and is isolated in $p$ by minimality of the length of $p$.
\end{rem}

\begin{rem}\label{difference}
First, in \cite[Definition 3.1]{HO13}, whenever $f,g\in G$ are in the same $H_\lambda$-coset $xH_\lambda$ for some $\lambda\in\Lambda$, $xH_\lambda$ is included in $S(f,g;D)$, but in our Definition \ref{separating cosets}, $S(f,g;D)$ can be empty even in this case. Secondly, in \cite[Definition 3.1]{HO13}, separating cosets are considered for each subgroup $H_\lambda$ separately, being denoted by $S_\lambda(f,g;D)$, but in our Definition \ref{separating cosets}, we consider separating cosets of all subgroups from the collection $\{H_\lambda\}_{\lambda\in\Lambda}$ all together.
\end{rem}

The following lemma is immediate from the above definition.

\begin{lem}
For any $f,g,h\in G$ and any $\lambda\in\Lambda$, the following hold.
\begin{itemize}
    \item [(a)]
    $S(f,g;D)=S(g,f;D)$.
    \item[(b)]
    $S(hf,hg;D)=\{hxH_\lambda \mid xH_\lambda\in S(f,g;D)\}$.
\end{itemize}
\end{lem}

We list some results on separating cosets so that readers have a better understanding.

\begin{lem}{\rm (cf. \cite[Lemma 4.5]{Osin16})}\label{lem:lem4.5 osin}
For any $f,g\in G$ and any $xH_\lambda \in S(f,g;D)$, every path in $\Gamma(G,X\sqcup\H)$ connecting $f$ to $g$ and composed of at most 2 geodesic segments penetrates $xH_\lambda$.
\end{lem}

The following lemma makes $S(f,g;D)$ into a totally ordered set.

\begin{lem}\emph{\cite[Lemma 4.6]{Osin16}} \label{lem:lem4.6 osin}
Let $f,g\in G$ and suppose that a geodesic path $p$ from $f$ to $g$ penetrates a coset $xH_\lambda$ for some $\lambda\in \Lambda$ and decomposes into $p=p_1ap_2$, where $p_1,p_2$ are possibly trivial and $a$ is a component corresponding to $xH_\lambda$. Then, we have $d_{X\cup\mathcal{H}}(f,a_-)=d_{X\cup\mathcal{H}}(f,xH_\lambda)$.
\end{lem}

\begin{defn}\cite[Definition 4.7]{Osin16} \label{def:Def4.7 Osin}
Given any $f,g\in G$, a relation $\preceq$ on the set $S(f,g;D)$ is defined as follows: for any $C_1,C_2 \in S(f,g;D)$,
\[C_1 \preceq C_2 \iff d_{X\cup\mathcal{H}}(f,C_1)\le d_{X\cup\mathcal{H}}(f,C_2).\]
\end{defn}

\begin{lem}\emph{\cite[Lemma 4.8]{Osin16}}\label{lem:lem4.8 osin}
    For any $f,g\in G$, the relation $\preceq$ is a linear order on $S(f,g;D)$ and any geodesic path $p$ in $\Gamma(G,X\sqcup\H)$ from $f$ to $g$ penetrates all $(f,g;D)$-separating cosets according to the order $\preceq$. In particular, $S(f,g;D)$ is finite. That is,
    \[
    S(f,g;D)=\{C_1 \preceq C_2 \preceq \cdots \preceq C_n\}
    \]
    for some $n\in\NN$ and $p$ decomposes into $p=p_1a_1\cdots p_na_np_{n+1}$, where $a_i$ is the component of $p$ corresponding to $C_i$ for each $i\in\{1,\cdots, n\}$.
\end{lem}

\subsection{Acylindrically hyperbolic group}
Theorem \ref{thm:AH group} is a simplified version of \cite[Theorem 1.2]{Osin16}. For more details on acylindrically hyperbolic groups, readers are referred to \cite{Osin16}.

\begin{thm} \label{thm:AH group}
    For any group $G$, the following conditions are equivalent.
    \begin{itemize}
        \item[$(AH_1)$]
        There exists a generating set $X$ of $G$ such that the corresponding Cayley graph $\Gamma(G,X)$ is hyperbolic, $|\partial\Gamma(G,X)|>2$, and the natural action of $G$ on $\Gamma(G,X)$ is acylindrical.
        \item[$(AH_4)$]
        $G$ contains a proper infinite hyperbolically embedded subgroup.
    \end{itemize}
\end{thm}

\begin{defn} \label{def:AH group}
    A group $G$ is called \emph{acylindrically hyperbolic}, if $G$ satisfies the equivalent conditions $(AH_1)$ and $(AH_4)$ from Theorem \ref{thm:AH group}.
\end{defn}

The following results were obtained in \cite[Section 5]{Osin16} in proving the implication $(AH_4)\Rightarrow(AH_1)$ in Theorem \ref{thm:AH group}. In Theorem \ref{thm:osin} below, we consider separating cosets for $\{H_\lambda\}_{\lambda\in\Lambda} \inj_h (G,X)$ and the metric $d_{\Gamma(G,Y\sqcup\H)}(\cdot,\cdot)$ is denoted by $d_{Y\cup\H}(\cdot,\cdot)$ for brevity.

\begin{thm}[{\rm cf. \cite[Theorem 5.4, Lemma 5.10]{Osin16}}]\label{thm:osin}
Suppose that $G$ is a group, $X$ is a subset of $G$, and $\{H_\lambda\}_{\lambda\in\Lambda}$ is a collection of subgroups of $G$ hyperbolically embedded in $(G,X)$. Let $C>0$ as in Proposition \ref{prop:C} and let $D>0$ satisfy $D \ge 3C$ as in \eqref{eq:D>3C}. We also define the subset $Y$ of $G$ by
\begin{align}\label{eq:Y}
    Y=\{y \in G \mid S(1,y;D)= \emptyset\}.
\end{align}
Then, we have for any $f,g \in G$,
\begin{equation}\label{eq:qi}
    \frac{1}{2}(d_{Y\cup\H}(f,g)-1) \le |S(f,g;D)| \le 3 d_{Y\cup\H}(f,g).
\end{equation}
If in addition $\Lambda$ is finite, then the following hold.
    \begin{itemize}
        \item[(a)]
        $\{H_\lambda\}_{\lambda\in\Lambda} \inj_h (G,Y)$.
        \item[(b)]
        The action of $G$ on $\Gamma(G,Y\sqcup\H)$ is acylindrical.
    \end{itemize}
\end{thm}

\begin{rem}\label{rem:finite collection}
    In \cite[Section 5]{Osin16}, the condition $|\Lambda|<\infty$ is assumed in all lemmas for proving $(AH_4)\Rightarrow(AH_1)$ including \cite[Lemma 5.10]{Osin16}, which corresponds to the inequality \eqref{eq:qi} in Theorem \ref{thm:osin}. However, the condition $|\Lambda|<\infty$ is not used in the proof of \cite[Lemma 5.10]{Osin16}, so we omit it in \eqref{eq:qi} for our discussion in Section \ref{sec:Gromov} and Section \ref{sec:appendix}. Actually, the condition $|\Lambda|<\infty$ is not used in the proof of \cite[Theorem 5.4 (a)]{Osin16} either.
\end{rem}

\begin{lem}{\rm \cite[Lemma 5.12]{Osin16}}\label{lem:non-elementary}
    Let $G$ be a group, $H$ a subgroup of $G$, $X$ a subset of $G$. If $H$ is non-degenerate (i.e. $H$ is a proper infinite subgroup of $G$) and $H\inj_h (G,X)$, then the action of $G$ on $\Gamma(G,X\sqcup H)$ is non-elementary (i.e. $|\partial \Gamma(G,X\sqcup H)|>2$).
\end{lem}


\section{Path representatives of the Gromov boundary} \label{sec:Gromov}

Throughout this section, suppose that $G$ is a group, $X$ is a subset of $G$, and $\{H_\lambda\}_{\lambda\in\Lambda}$ is a collection of subgroups of $G$ hyperbolically embedded in $(G,X)$. Let $C>0$ as in Proposition \ref{prop:C} and fix $D>0$ satisfying $D\ge 3C$ as in \eqref{eq:D>3C}. We also define the subset $Y$ of $G$ by $Y=\{y \in G \mid S(1,y;D)= \emptyset\}$ as in Theorem \ref{thm:osin}.

In this section, we will show that elements of the Gromov boundary of the Cayley graph $\Gamma(G,Y\sqcup\H)$ are represented by nice geodesic rays in $\Gamma(G,X\sqcup\H)$ (see Proposition \ref{prop:existence}). The nice geodesic rays are characterized by penetrating infinitely many cosets deeply enough (see Lemma \ref{lem:seq to infinity}). By using these path representatives of boundary points, we will extend the notion of Hull-Osin's separating cosets to allow a point in the Gromov boundary (see Definition \ref{def:sep coset for vertex and boundry points}). We will also investigate the relation between the path representatives and the topology of $\partial\Gamma(G,Y\sqcup\H)$ (see Proposition \ref{prop:for cts} and Proposition \ref{prop:cts from X to Y}).

For brevity, we will denote $d_{\Gamma(G,X\sqcup\H)}(\cdot,\cdot)$ and $(\cdot,\cdot)_\cdot^{\Gamma(G,X\sqcup\H)}$ by $d_{X\cup\H}(\cdot,\cdot)$ and $(\cdot,\cdot)_\cdot^{X\cup\H}$ respectively (see \eqref{eq:gromov product}). This will be the same for $\Gamma(G,Y\sqcup\H)$ as well. We also emphasize that we will consider separating cosets and relative metrics for $\{H_\lambda\}_{\lambda\in\Lambda} \inj_h (G,X)$, hence we use the notations $S(\cdot,\cdot;D)$ and $\hd_{\lambda}(\cdot,\cdot)$ without including $X$ in them.

As preparation of our discussion, we list auxiliary results that immediately follow from Section \ref{subsec:HE subgroups} but have not been recorded explicitly. They're from Lemma \ref{lem:inclusion} up to Lemma \ref{lem:sep coset inclustion}.

\begin{lem} \label{lem:inclusion}
    Let $p$ be a geodesic path in $\Gamma(G,X\sqcup\H)$ between two vertices and $q$ be a subpath of $p$, then we have $S(q_-,q_+;D) \subset S(p_-,p_+;D)$.
\end{lem}

\begin{proof}
    Let $p=p_1qp_2$ be decomposition of $p$ into subpaths $p_1$, $q$, and $p_2$. For any $H_\lambda$-coset $xH_\lambda \in S(q_-,q_+;D)$, there exists a geodesic path $\alpha$ in $\Gamma(G,X\sqcup\H)$ from $q_-$ to $q_+$ that essentially penetrates $xH_\lambda$. Since $p_1\alpha p_2$ is a geodesic path in $\Gamma(G,X\sqcup\H)$ from $p_-$ to $p_+$ that essentially penetrates $xH_\lambda$, we have $xH_\lambda\in S(p_-,p_+;D)$.
\end{proof}

Lemma \ref{lem:bounded by 3C} means that if two geodesic paths from the same point penetrate the same coset, then their entrance points are close.

\begin{lem}\label{lem:bounded by 3C}
    Let $o\in G$ and suppose that $B$ is an $H_\lambda$-coset for some $\lambda\in\Lambda$ and that $p,q$ are (possibly infinite) geodesic paths from $o$ in $\Gamma(G,X\sqcup\H)$ that penetrate $B$. Then, we have $\hd_{\lambda}(p_{in}(B),q_{in}(B)) \le 3C$.
\end{lem}

\begin{proof}
    Let $x=p_{in}(B)$ and $y=q_{in}(B)$ for brevity and let $e$ be the edge in $\Gamma(G,X\sqcup\H)$ from $x$ to $y$ whose label is in $H_\lambda$. Since $p,q$ are geodesic in $\Gamma(G,X\sqcup\H)$ and penetrates $B$, $e$ is an isolated component in the geodesic triangle $p_{[o,x]}e(q_{[o,y]})^{-1}$ by Remark \ref{rem:component of geodesic}. This implies $\hd_\lambda(x,y)\le 3C$ by Proposition \ref{prop:C}.
\end{proof}

Lemma \ref{lem:distance of cosets} means that the distance between two cosets can be measured by a geodesic path penetrating both of them.

\begin{lem}\label{lem:distance of cosets}
    Suppose that $p$ is a (possibly infinite) geodesic path in $\Gamma(G,X\sqcup\H)$ from $p_-\in G$. If $p$ penetrates two distinct cosets $C_0,C_1$ satisfying $d_{X\cup\H}(p_-,C_0)< d_{X\cup\H}(p_-,C_1)$, then we have
    \[
    d_{X\cup\H}(p_{out}(C_0),p_{in}(C_1))=d_{X\cup\H}(C_0,C_1).
    \]
\end{lem}

\begin{proof}
    Let $p=p_1ap_2bp_3$ be decomposition of $p$ into subpaths such that $a,b$ are components of $p$ corresponding to $C_0,C_1$ respectively. By $a_+\in C_0$ and $b_-\in C_1$, we have $d_{X\cup\H}(C_0,C_1) \le d_{X\cup\H}(a_+,b_-)$. Suppose for contradiction that there exist $x\in C_0, y\in C_1$ such that $d_{X\cup\H}(x,y) < d_{X\cup\H}(a_+,b_-)$. By $x,a_-\in C_0$ and $y,b_+\in C_1$, we have $d_{X\cup\H}(a_-,x)\le 1$ and $d_{X\cup\H}(y,b_+) \le 1$. This implies
    \begin{align*}
            d_{X\cup\H}(a_-,b_+)
        &\le d_{X\cup\H}(a_-,x)+d_{X\cup\H}(x,y)+d_{X\cup\H}(y,b_+)\\
        &< 1+d_{X\cup\H}(a_+,b_-)+1
        =d_{X\cup\H}(a_-,b_+),
    \end{align*}
    which is a contradiction. Hence, we also have $d_{X\cup\H}(a_+,b_-) \le d_{X\cup\H}(C_0,C_1)$.
\end{proof}

\begin{figure}[htbp]

\begin{minipage}[c]{0.5\hsize}
\begin{center}
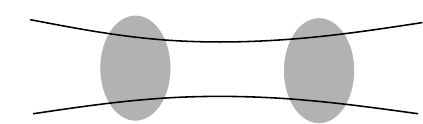
 \caption{Proof of Lemma \ref{lem:bounded by 4C}} 
 \label{Fig:Lem3_4}
\end{center}
\end{minipage} 
\hfill 
\begin{minipage}[c]{0.5\hsize}
\begin{center}
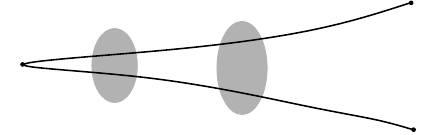
 \caption{Proof of Lemma \ref{lem:sep coset inclustion}}
 \label{Fig:Lem3_5}
\end{center}
\end{minipage} 

\end{figure}

Lemma \ref{lem:bounded by 4C} is analogous to Lemma \ref{lem:bounded by 3C}.

\begin{lem}\label{lem:bounded by 4C}
    Suppose that $C_0,C_1$ are cosets of $H_{\lambda_0}, H_{\lambda_1}$ respectively with $C_0\neq C_1$ and that $p,q$ are (possibly infinite) geodesic paths in $\Gamma(G,X\sqcup\H)$ from $p_-,q_- \in G$ respectively that penetrate $C_0$ and $C_1$ satisfying $d_{X\cup\H}(p_-,C_0) < d_{X\cup\H}(p_-,C_1)$ and $d_{X\cup\H}(q_-,C_0) < d_{X\cup\H}(q_-,C_1)$. Then, we have $\hd_{\lambda_0}(p_{out}(C_0),q_{out}(C_0)) \le 4C$ and $\hd_{\lambda_1}(p_{in}(C_1),q_{in}(C_1)) \le 4C$.
\end{lem}

\begin{proof}
     Let $x_0=p_{out}(C_0), x_1=p_{in}(C_1), y_0=q_{out}(C_0), y_1=q_{in}(C_1)$ for brevity and let $e_0,e_1$ be edges in $\Gamma(G,X\sqcup\H)$ such that $e_0$ is from $x_0$ to $y_0$ with its label in $H_{\lambda_0}$ and $e_1$ is from $x_1$ to $y_1$ with its label in $H_{\lambda_1}$. Since the subpaths $p_{[x_0,x_1]},q_{[y_0,y_1]}$ don't penetrate $C_0$ nor $C_1$ by Remark \ref{rem:component of geodesic} and we have $C_0\neq C_1$, $e_0$ and $e_1$ are isolated components of the geodesic quadrilateral $e_0 q_{[y_0,y_1]} e_1^{-1} (p_{[x_0,x_1]})^{-1}$. This implies $\hd_{\lambda_0}(x_0,y_0)\le 4C$ and $\hd_{\lambda_1}(x_1,y_1)\le 4C$ by Proposition \ref{prop:C}.
\end{proof}

Lemma \ref{lem:sep coset inclustion} is useful to find separating cosets of a pair of elements in $G$.

\begin{lem} \label{lem:sep coset inclustion}
    Let $o,x,y\in G$ and $S(o,x;D)=\{C_1\preceq C_2\preceq \cdots\preceq C_n\}$. If a geodesic path $q$ in $\Gamma(G,X\sqcup\H)$ from $o$ to $y$ penetrates $C_i$ for some $i\in\{1,\cdots,n\}$, then we have $C_j\in S(o,y;D)$ for any $j$ with $j<i$.
\end{lem}

\begin{proof}
    Let $q$ penetrate $C_i$ and $j\in\NN$ satisfy $j<i$. By $C_j\in S(o,x;D)$, their exists a geodesic path $p$ in $\Gamma(G,X\sqcup\H)$ from $o$ to $x$ that essentially penetrates $C_j$. Note that $p$ penetrates $C_i$ by Lemma \ref{lem:lem4.8 osin}. Let $p=p_1 a p_2$, $q=q_1 b q_2$ be decomposition of $p,q$ into subpaths such that $a,b$ are $H_\lambda$-component of $p,q$ corresponding to $C_i$ respectively. Let $e$ be the edge from $a_-$ to $b_+$ in $\Gamma(G,X\sqcup\H)$ whose label is in $H_\lambda$. Since we have $d_{X\cup\H}(o,a_-)=d_{X\cup\H}(o,b_-)=d_{X\cup\H}(o,C_i)$ by Lemma \ref{lem:lem4.6 osin}, the path $p_1eq_2$ from $o$ to $y$ is geodesic in $\Gamma(G,X\sqcup\H)$ and essentially penetrates $C_j$. This implies $C_j \in S(o,y;D)$.
\end{proof}

First of all, we verify hyperbolicity of $\Gamma(G,Y\sqcup\H)$. Lemma \ref{lem:Haus} is straightforward from \cite[Proposition 3.1]{Bow}, which is stated below in a simplified way, but we write down the proof for completeness. Also, Lemma \ref{lem:Haus} (a) is actually known by \cite[Lemma 5.6]{Osin16} since its proof doesn't use the condition $|\Lambda|<\infty$ (see Remark \ref{rem:finite collection}). Lemma \ref{lem:Haus} (b) is new and plays an important role in this paper together with the inequality \eqref{eq:qi} in Theorem \ref{thm:osin}.

\begin{prop}[{\rm cf. \cite[Proposition 3.1]{Bow}}]\label{prop:Bow}
Given $h\ge0$, there exists $k(h)\ge0$ with the following property. Suppose that $\Gamma$ is a connected graph and that for each pair of vertices $x,y\in\Gamma$, we have associated a connected subgraph $\mathcal{L}(x,y)$ of $\Gamma$ with $x,y \in \mathcal{L}(x,y)$ satisfying (1) and (2) below. (Here, we define $\mathcal{N}(A,h)=\{v\in\Gamma \mid \exists w\in A {\rm ~s.t.~} d_{\Gamma}(v,w)\le h \}$.)
\begin{itemize}
    \item[(1)]
    For any vertices $x,y,z\in\Gamma$, $\mathcal{L}(x,y) \subset \mathcal{N}(\mathcal{L}(x,z)\cup \mathcal{L}(z,y),h)$.
    \item[(2)]
    For any vertices $x,y \in\Gamma$ with $d_{\Gamma}(x,y) \le 1$, the diameter of $\mathcal{L}(x,y)$ in $\Gamma$ is at most $h$.
\end{itemize}
Then, $\Gamma$ is $k(h)$-hyperbolic and for any two vertices $x,y\in \Gamma$, the Hausdorff distance between $\mathcal{L}(x,y)$ and any geodesic path in $\Gamma$ from $x$ to $y$ is bounded above by $k(h)$.
\end{prop}

\begin{lem} \label{lem:Haus}
The following hold.
\begin{itemize}
    \item[(a)]
    We have $X\subset Y$ and $\Gamma(G,Y\sqcup\H)$ is hyperbolic.
    \item[(b)]
    There exists $M_X \in\NN$ such that for any $x,y \in G$, any geodesic path $\alpha$ in $\Gamma(G,X\sqcup\H)$ from $x$ to $y$, and any geodesic path $\beta$ in $\Gamma(G, Y\sqcup\H)$ from $x$ to $y$, the Hausforff distance betwen $\alpha$ and $\beta$ in $\Gamma(G,Y\sqcup\H)$ is bounded above by $M_X$.
\end{itemize}
\end{lem}

\begin{proof}
    If $x\in X$ and $x\neq 1$, then the edge in $\Gamma(G,X\sqcup\H)$ from $1$ to $x$ with the label $x\in X$ is geodesic in $\Gamma(G,X\sqcup\H)$. Since this path consisting of one edge has no $H_\lambda$-component for any $\lambda\in\Lambda$, we have $S(1,x;D)=\emptyset$ by Lemma \ref{lem:lem4.8 osin}. This implies $x\in Y$, hence $X\subset Y$. We will check the two conditions in Proposition \ref{prop:Bow} considering $\Gamma=\Gamma(G,Y\sqcup\H)$. Note that $\Gamma(G,X\sqcup\H)$ is a subgraph of $\Gamma(G,Y\sqcup\H)$ by $X\subset Y$. For each pair $(x,y)$ of elements of $G$, fix a geodesic path $\gamma_{x,y}$ in $\Gamma(G,X\sqcup\H)$ from $x$ to $y$ and define $\mathcal{L}(x,y)=\gamma_{x,y}$. Since $\Gamma(G,X\sqcup\H)$ is $\delta_X$-hyperbolic with $\delta_X\in \NN$ (see Definition \ref{def:hyp emb} (1)), Proposition \ref{prop:Bow} (1) is satisfied with $h=\delta_X$. Next, let $x,y\in G$ satisfy $d_{Y\cup\H}(x,y)\le 1$, then $S(x,y;D)=\emptyset$. Hence, for any vertex $z\in \gamma_{x,y}$, we have $S(x,z;D)=\emptyset$ by Lemma \ref{lem:inclusion}. This implies $d_{Y\cup\H}(x,z)\le 1$. Hence, the diameter of $p_{x,y}$ in $\Gamma(G,Y\sqcup\H)$ is at most 2, which verifies Proposition \ref{prop:Bow} (2). Since both conditions in Proposition \ref{prop:Bow} are satisfied with $h=\max\{\delta_X,2\}$, the graph $\Gamma(G,Y\sqcup\H)$ is $k(h)$-hyperbolic. Also, let $x,y \in G$ and let $\alpha$ be a geodesic path in $\Gamma(G,X\sqcup\H)$ from $x$ to $y$, and $\beta$ be a geodesic path in $\Gamma(G, Y\sqcup\H)$ from $x$ to $y$. By Proposition \ref{prop:Bow}, the Hausdorff distance between $\beta$ and $\gamma_{x,y}$ in $\Gamma(G,Y\sqcup\H)$ is at most $k(h)$ and by $\delta_X$-hyperbolicity of $\Gamma(G,X\sqcup\H)$, the Hausdorff distance between $\alpha$ and $\gamma_{x,y}$ in $\Gamma(G,Y\sqcup\H)$ is at most $\delta_X$. Thus, the statement (b) holds with $M_X=k(h)+\delta
    _X$.
\end{proof}

In the remainder of this section, let $\Gamma(G,Y\sqcup\H)$ (resp. $\Gamma(G,X\sqcup\H)$) be $\delta_Y$-hyperbolic (resp. $\delta_X$-hyperbolic) with $\delta_X,\delta_Y \in \NN$.

\begin{rem}
    Note that $\delta_Y$ and $M_X$ depend only on $X$ by the proof of Lemma \ref{lem:Haus}.
\end{rem}

The point of Theorem \ref{thm:osin} \eqref{eq:qi} and Lemma \ref{lem:Haus} (b) is that we can deal with geodesic paths in $\Gamma(G,X\sqcup\H)$ as if they are quasi-geodesic in $\Gamma(G,Y\sqcup\H)$, though they're actually not. This will become clear by results from Lemma \ref{lem:gromov prod in Y} up to Lemma \ref{lem:penetration in triangle} below. Lemma \ref{lem:gromov prod in Y} is mostly applied to two elements $g_1,g_2$ except in the proof of Proposition \ref{prop:existence}.

\begin{lem}\label{lem:gromov prod in Y}
    Let $R\in\NN$, $(g_i)_{i\in\NN},o \in G$ and suppose that $p_i$ is a geodesic path in $\Gamma(G,X\sqcup\H)$ from $o$ to $g_i$ for each $i\in\NN$. If $(g_i,g_j)_o^{Y\cup\H}\ge R$ for any $i,j\in\NN$, then there exist vertices $v_i \in p_i$ for each $i\in\NN$ such that for any $i,j\in\NN$, we have
    \[
    d_{Y\cup\H}(v_i,v_j)\le \delta_Y+2M_X
    ~~~and~~~
    d_{Y\cup\H}(o,v_i) \ge R-M_X.
    \]
\end{lem}

\begin{proof}
    For each $i \in\NN$, let $[o,g_i]$ be a geodesic path in $\Gamma(G,Y\sqcup\H)$ from $o$ to $g_i$ and $w_i \in [o,g_i]$ be a vertex satisfying $d_{Y\cup \H} (o,w_i)=R$. By $(g_i,g_j)_o^{Y\cup\H}\ge R$, we have $d_{Y\cup\H}(w_i,w_j) \le \delta_Y$ for any $i,j \in \NN$. By Lemma \ref{lem:Haus} (b), for each $i\in\NN$, there exists a vertex $v_i \in p_i$ such that $d_{Y\cup\H}(v_i,w_i) \le M_X$. We have for any $i,j\in\NN$,
    \begin{align*}
        &d_{Y\cup\H}(v_i,v_j)
        \le d_{Y\cup\H}(v_i,w_i) + d_{Y\cup\H}(w_i,w_j) + d_{Y\cup\H}(w_j,v_j)
        \le \delta_Y + 2M_X \\
        {\rm and~~~~~}
        &d_{Y\cup\H}(o, v_i)
        \ge d_{Y\cup\H}(o, w_i) - d_{Y\cup\H}(w_i,v_i)
        \ge R-M_X.
    \end{align*}
\end{proof}

Lemma \ref{lem:estimate gromov prod in Y} below can be considered as the converse of Lemma \ref{lem:gromov prod in Y}.

\begin{lem} \label{lem:estimate gromov prod in Y}
    Let $o,x,y \in G$ and suppose that $p,q$ are geodesic paths in $\Gamma(G,X\sqcup\H)$ such that $p$ is from $o$ to $x$ and $q$ is from $o$ to $y$. If there exist vertices $v\in p$ and $w\in q$ such that $d_{Y\cup\H}(o,v) \ge R$ and $d_{Y\cup\H}(v,w) \le K$ with some $R,K\in\NN$, then we have $(x,y)_o^{Y\cup\H}\ge R-(K+3M_X) - 2\delta_Y$.
\end{lem}

\begin{proof}
    Take geodesic paths $[o,x],[o,y]$ in $\Gamma(G,Y\sqcup\H)$. By Lemma \ref{lem:Haus} (b), there exist vertices $a\in [o,x], b \in [o,y]$ such that $d_{Y\cup\H}(a,v) \le M_X$ and $d_{Y\cup\H}(b,w) \le M_X$. We have
    \begin{align*}
        d_{Y\cup\H}(a,b)
        &\le d_{Y\cup\H}(a,v) + d_{Y\cup\H}(v,w) + d_{Y\cup\H}(w,b)
        \le K + 2M_X, \\
        d_{Y\cup\H}(o,a)
        &\ge d_{Y\cup\H}(o,v)- d_{Y\cup\H}(a,v)
        \ge R-M_X, \\
        {\rm and~~~~~}
        d_{Y\cup\H}(o,b)
        &\ge d_{Y\cup\H}(o,v)- (d_{Y\cup\H}(v,w) + d_{Y\cup\H}(w,b))
        \ge R-(K+M_X).
    \end{align*}
    This implies $(a,b)_o^{Y\cup\H} \ge d_{Y\cup\H}(o,a)-d_{Y\cup\H}(a,b) \ge R-(K+3M_X)$. Note $(x,a)_o^{Y\cup\H}=d_{Y\cup\H}(o,a)$ and $(y,b)_o^{Y\cup\H}=d_{Y\cup\H}(o,b)$ since $[o,a], [o,b]$ are geodesic in $\Gamma(Y\cup\H)$. Hence,
    \begin{align*}
        (x,y)_o^{Y\cup\H}
        &\ge \min\{(x,a)_o^{Y\cup\H}, (a,b)_o^{Y\cup\H}, (b,y)_o^{Y\cup\H} \}-2\delta_Y \\
        &\ge R-(K+3M_X) - 2\delta_Y.
    \end{align*}
\end{proof}

Lemma \ref{lem:penetration in triangle} describes slim triangle property of $\Gamma(G,Y\sqcup\H)$ using separating cosets.

\begin{lem}\label{lem:penetration in triangle}
    Let $o,x,y \in G$ and $S(o,x;D)=\{C_1\preceq C_2\preceq \cdots\preceq C_n\}$. If $q$ is a geodesic path in $\Gamma(G,X\sqcup\H)$ from $o$ to $y$ and $i\in\{1,\cdots,n\}$ satisfies $3d_{Y\cup\H}(x,y)+1 \le n-i$, then $q$ penetrates $C_i$.
\end{lem}

\begin{proof}
     Let $i \in \{1,\cdots,n\}$ satisfy $3d_{Y\cup\H}(x,y)+1 \le n-i$ and take a geodesic path $r$ in $\Gamma(G,X\sqcup\H)$ from $y$ to $x$. Since the path $qr$ is from $o$ to $x$ and composed of two geodesic segments, one of $q$ or $r$ penetrates $C_i$ by Lemma \ref{lem:lem4.5 osin}. Suppose for contradiction that $r$ penetrates $C_i$. This implies $C_j\in S(x,y;D)$ for any $j\in \{i+1,\cdots,n\}$ by applying Lemma \ref{lem:sep coset inclustion} to $r^{-1}$ and $S(x,o;D)=\{C_n \preceq \cdots \preceq C_1\}$. This and \eqref{eq:qi} imply
     \[
     3d_{Y\cup\H}(x,y)+1 \le n-i \le |S(x,y;D)| \le 3d_{Y\cup\H}(x,y).
     \]
     This is a contradiction. Hence, $q$ penetrates $C_i$.
\end{proof}

Hull-Osin's separating cosets have been defined only for a pair of group elements. Now, we define separating cosets for geodesic rays. This notion is useful to clarify nice geodesic rays in $\Gamma(G,X\sqcup\H)$ mentioned at the beginning of this section.

\begin{defn}\label{def:S(gamma;D)}
    Let $\gamma=(x_0,x_1,\cdots)$ be a geodesic ray in $\Gamma(G,X\sqcup\H)$. We define
    \[
    S(\gamma;D)=\bigcup_{n=0}^\infty S(x_0,x_n;D)
    \]
    and call an element in $S(\gamma;D)$ a $(\gamma;D)$-\emph{separating coset}.
\end{defn}

\begin{rem}\label{rem:increasing}
    By Lemma \ref{lem:inclusion}, we have $S(x_0,x_{n-1};D) \subset S(x_0,x_n;D)$ for any $n\in\NN$. This implies
    \[
    |S(\gamma;D)|=\lim_{n\to\infty} |S(x_0,x_{n};D)|=\sup_{n\in\NN} |S(x_0,x_{n};D)|.
    \]
\end{rem}

We collect basic properties of separating cosets for geodesic rays from Lemma \ref{lem:geod ray penetrate} up to Lemma \ref{lem: sep coset ordered}.

\begin{lem} \label{lem:geod ray penetrate}
    Suppose that $\gamma$ is a geodesic ray in $\Gamma(G,X\sqcup\H)$, then $\gamma$ penetrates all $(\gamma;D)$-separating cosets exactly once.
\end{lem}

\begin{proof}
    Let $\gamma=(x_0,x_1,\cdots)$ and $B\in S(\gamma;D)$. There exists $n\in \NN$ such that $B\in S(x_0,x_n;D)$. By Lemma \ref{lem:lem4.8 osin}, $\gamma_{[x_0,x_n]}$ penetrates $B$, hence so does $\gamma$. If $\gamma$ penetrates an $H_\lambda$-coset more than once, then $\gamma$ can be shortened by an edge whose label is in $H_\lambda$, which contradicts that $\gamma$ is geodesic in $\Gamma(G,X\sqcup\H)$.
\end{proof}

Note that in Lemma \ref{lem:geod ray penetrate}, the fact that $\gamma$ penetrates all $(\gamma;D)$-separating cosets doesn't trivially follow from the definition of $S(\gamma;D)$, because by Definition \ref{def:S(gamma;D)}, an $H_\lambda$-coset $C$ is in $S(\gamma;D)$ if and only if there exist $n\in\NN$ and a geodesic path $p$ in $\Gamma(G,X\sqcup\H)$ from $x_0$ to $x_n$ such that $p$ essentially penetrates $C$, and $p$ may not be a subpath of $\gamma$.

Lemma \ref{lem:distance of S gamma} is analogous to Lemma \ref{lem:lem4.6 osin}.

\begin{lem} \label{lem:distance of S gamma}
    Suppose that a geodesic ray $\gamma$ in $\Gamma(G,X\sqcup\H)$ penetrates an $H_\lambda$-coset $B$, then we have $d_{X\cup\H}(\gamma_-,\gamma_{in}(B))=d_{X\cup\H}(\gamma_-,B)$.
\end{lem}

\begin{proof}
    Let $\gamma=(x_0,x_1,\cdots)$. Take $n\in \NN$ satisfying $d_{X\cup\H}(\gamma_-,\gamma_{out}(B))<d_{X\cup\H}(\gamma_-,x_n)$, then we have $d_{X\cup\H}(\gamma_-,B)=d_{X\cup\H}(\gamma_-,(\gamma_{[x_0,x_n]})_{in}(B))=d_{X\cup\H}(\gamma_-,\gamma_{in}(B))$ by Lemma \ref{lem:lem4.6 osin}.
\end{proof}

As in Definition \ref{def:Def4.7 Osin}, we can align separating cosets for a geodesic ray based on the order of their penetration.

\begin{defn}\label{def:order on S gamma}
    Given a geodesic ray $\gamma$ in $\Gamma(G,X\sqcup\H)$, a relation $\preceq$ on the set $S(\gamma;D)$ is defined as follows: for any $C_1,C_2 \in S(\gamma;D)$,
\[C_1 \preceq C_2 \iff d_{X\cup\H}(\gamma_-,C_1)\le d_{X\cup\H}(\gamma_-,C_2).\]
\end{defn}

\begin{rem}
     By Lemma \ref{lem:geod ray penetrate} and Lemma \ref{lem:distance of S gamma}, the relation $\preceq$ in Definition \ref{def:order on S gamma} is a well-order on $S(\gamma;D)$. We will denote $S(\gamma;D)=\{C_1\preceq C_2\preceq \cdots\}$ considering this order.
\end{rem}

In Lemma \ref{lem: sep coset ordered}, given a finite collection of separating cosets of a geodesic ray, we find how long subpath of the geodesic ray we have to take to contain them.

\begin{lem}\label{lem: sep coset ordered}
    Let $\gamma=(x_0,x_1,\cdots)$ be a geodesic ray in $\Gamma(G,X\sqcup\H)$ and $S(\gamma;D)=\{C_1\preceq C_2\preceq \cdots\}$. For any $N,\ell\in\NN$ satisfying $d_{X\cup\H}(x_0,\gamma_{out}(C_{N+1})) \le d_{X\cup\H}(x_0,x_\ell)$, we have $\{C_1,\cdots,C_N\} \subset S(x_0,x_\ell;D)$. Moreover, letting $S(x_0,x_\ell;D)=\{B_1\preceq\cdots\preceq B_N\preceq\cdots\}$, we have $B_n=C_n$ for any $n\in\{1,\cdots,N\}$.
\end{lem}

\begin{proof}
    By $d_{X\cup\H}(x_0,\gamma_{out}(C_{N+1})) \le d_{X\cup\H}(x_0,x_\ell)$, the subpath $\gamma_{[x_0,x_\ell]}$ penetrates $C_{N+1}$. Take $k\in\NN$ such that $\{C_1,\cdots,C_{N+1}\}\subset S(x_0,x_k;D)$. By applying Lemma \ref{lem:sep coset inclustion} to $x_0,x_k,\gamma_{[x_0,x_\ell]}$, we have $\{C_1,\cdots,C_N\}\subset S(x_0,x_\ell;D)$. Let $S(x_0,x_\ell;D)=\{B_1\preceq\cdots\preceq B_N\preceq\cdots\}$. We have $B_1\preceq C_1$ by $C_1\in S(x_0,x_\ell;D)$ and by minimality of $B_1$ in $(S(x_0,x_\ell;D),\preceq)$. On the other hand, by minimality of $C_1$ in $(S(\gamma;D),\preceq)$, we also have $C_1\preceq B_1$, hence $B_1=C_1$. Repeating the same argument inductively, we can see $B_n=C_n$ for any $n\in\{1,\cdots,N\}$.  
\end{proof}

We can now characterize nice geodesic rays in $\Gamma(G,X\sqcup\H)$ and study their properties. Lemma \ref{lem:seq to infinity} lists several equivalent conditions of Definition \ref{def:nice geodesic ray}.

\begin{defn}\label{def:nice geodesic ray}
    We say that a geodesic ray $\gamma=(x_0,x_1,\cdots)$ in $\Gamma(G,X\sqcup\H)$ \emph{converges to infinity in} $\Gamma(G,Y\sqcup\H)$, if the sequence $(x_n)_{n=1}^\infty$ in $G$ converges to infinity in $\Gamma(G,Y\sqcup\H)$.
\end{defn}

\begin{lem}\label{lem:seq to infinity}
    Suppose that $\gamma=(x_0,x_1,\cdots)$ is a geodesic ray in $\Gamma(G,X\sqcup\H)$, then (i)-(vi) below are all equivalent.
    \begin{itemize}
        \item[(i)]
        The geodesic ray $\gamma$ converges to infinity in $\Gamma(G,Y\sqcup\H)$.
        \item[(ii)] 
        There exists a subsequence $(x_{n_k})_{k=1}^\infty$ that converges to infinity in $\Gamma(G,Y\sqcup\H)$.
        \item[(iii)]
        $|S(\gamma;D)|=\infty$.
        \item[(iv)]
        There exists a subsequence $(x_{n_k})_{k=1}^\infty$ such that $\lim_{k\to\infty} |S(x_0,x_{n_k};D)| = \infty$.
        \item[(v)]
        $\lim_{n\to\infty} d_{Y\cup\H}(x_0,x_n) =\infty$.
        \item[(vi)]
        There exists a subsequence $(x_{n_k})_{k=1}^\infty$ such that $\lim_{k\to\infty} d_{Y\cup\H}(x_0,x_{n_k}) =\infty$.
    \end{itemize}
\end{lem}

\begin{proof}
    (i)$\Rightarrow$(ii) and (v)$\Rightarrow$(vi) are trivial. (iii)$\Rightarrow$(iv) follows from Remark \ref{rem:increasing}. (i)$\Rightarrow$(v) and (ii)$\Rightarrow$(vi) follow from the definition of the Gromov product (see \eqref{eq:gromov product}). (iii)$\Leftrightarrow$(v) and (iv)$\Leftrightarrow$(vi) follow from \eqref{eq:qi}. We are left to show (iv)$\Rightarrow$(iii) and (v)$\Rightarrow$(i).
    
    \underline{(iv)$\Rightarrow$(iii)}
    By Remark \ref{rem:increasing}, we have
    \[
    \lim_{k\to\infty} |S(x_0,x_{n_k};D)|=\sup_{k\in\NN} |S(x_0,x_{n_k};D)|
    =\sup_{n\in\NN} |S(x_0,x_n;D)|=\lim_{n\to\infty} |S(x_0,x_n;D)|.
    \]
    Thus, $\lim_{k\to\infty} |S(x_0,x_{n_k};D)| = \infty$ implies $|S(\gamma;D)|=\lim_{n\to\infty} |S(x_0,x_n;D)| = \infty$.

    \underline{(v)$\Rightarrow$(i)}
    For any $R\in\NN$, define $R_1=R+3M_X+2\delta_Y$. By (v), there exists $N\in\NN$ such that $d_{Y\cup\H}(x_0,x_N) \ge R_1$. Hence, for any $n,m\ge N$, we have $(x_n,x_m)_{x_0}^{Y\cup\H}\ge R_1-3M_X-2\delta_Y=R$ by Lemma \ref{lem:estimate gromov prod in Y} applied to $o=x_0, x=x_n, y=x_m, v=x_N, w=x_N$. This implies $\liminf_{n,m\to\infty}(x_n,x_m)_{x_0}^{Y\cup\H} \ge \inf_{n,m\ge N} (x_n,x_m)_{x_0}^{Y\cup\H} \ge R$ for any $R\in\NN$. Thus, $\lim_{n,m\to\infty}(x_n,x_m)_{x_0}^{Y\cup\H}=\infty$.
\end{proof}

In Definition \ref{def:Y-lim} below, we summarize notations related to limits of the nice geodesic rays, which we will use in what follows.

\begin{defn}\label{def:Y-lim}
    For a sequence $(x_n)_{n=1}^\infty$ of elements of $G$ that is convergent in $\Gamma(G,Y\sqcup\H)\cup\partial\Gamma(G,Y\sqcup\H)$, we denote its limit point by $Y\-\lim_{n\to\infty}x_n$. For a geodesic ray $\gamma=(x_0,x_1,\cdots)$ in $\Gamma(G,X\sqcup\H)$ such that the sequence $(x_n)_{n=1}^\infty$ is convergent in $\Gamma(G,Y\sqcup\H)\cup\partial\Gamma(G,Y\sqcup\H)$, we also denote its limit point by $Y\-\lim\gamma$. Note that when we write $Y\-\lim_{n\to\infty}x_n \in \partial\Gamma(G,Y\sqcup\H)$ for a sequence $(x_n)_{n=1}^\infty$ in $G$, it implicitly means that $(x_n)_{n=1}^\infty$ converges to infinity in $\Gamma(G,Y\sqcup\H)$ and its limit point in $\partial\Gamma(G,Y\sqcup\H)$ is $Y\-\lim_{n\to\infty}x_n$. This is the same for $Y\-\lim \gamma \in \partial\Gamma(G,Y\sqcup\H)$ as well. We use the notations $X\-\lim_{n\to\infty}x_n$ and $X\-\lim \gamma$ similarly.
\end{defn}

\begin{figure}[htbp]
  \begin{center}
 \hspace{0mm} 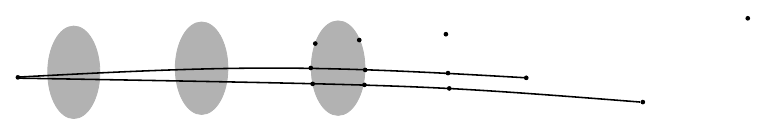
  \end{center}
   \vspace{-3mm}
  \caption{Proof of Proposition \ref{prop:existence}}
  \label{Prop3_22}
\end{figure}

We will now show that any element of $G$ and any point in $\partial\Gamma(G,Y\sqcup\H)$ can be connected by a geodesic ray in $\Gamma(G,X\sqcup\H)$, and by using this we will extend the notion of Hull-Osin's separating cosets to allow a point in the Gromov boundary. We emphasize that in Proposition \ref{prop:existence}, the path $\gamma$ below is not necessarily geodesic nor quasi-geodesic in $\Gamma(G,Y\sqcup\H)$.

\begin{prop}\label{prop:existence}
    For any $o\in G$ and any $\xi \in \partial\Gamma(G,Y\sqcup\H)$, there exists a geodesic ray $\gamma$ in $\Gamma(G,X\sqcup\H)$ from $o$ such that $Y\-\lim\gamma=\xi \in \partial \Gamma(G,Y\sqcup\H)$.
\end{prop}

\begin{proof}
Let $(g_n)_{n=1}^\infty$ be a sequence of elements of $G$ that converges to $\xi$.
For each $n\in\NN$, take a geodesic path $\gamma_n$ in $\Gamma(G,X\sqcup\H)$ from $o$ to $g_n$. Since $\lim_{i,j\to \infty} (g_i,g_j)_o^{Y\cup\H} = \infty$, there exists a subsequence $(g_{0k})_{k=1}^\infty$ of $(g_n)_{n=1}^\infty$ such that
\begin{equation}\label{eq:g0k}
    \inf_{k,\ell \in\NN}(g_{0k},g_{0\ell})_o^{Y\cup\H} \ge 2(3(\delta_Y + 2M_X)+4) + 1 + M_X.
\end{equation}
By Lemma \ref{lem:gromov prod in Y} and \eqref{eq:g0k}, there exist vertices $v_{0k} \in \gamma_{0k}$ for each $k \in\NN$ satisfying
$d_{Y\cup\H}(v_{0k},v_{0\ell}) \le \delta_Y + 2M_X$ and $d_{Y\cup\H}(o, v_{0\ell}) \ge 2(3(\delta_Y + 2M_X)+4) + 1$ for any $k,\ell\in\NN$. This and \eqref{eq:qi} imply for any $k,\ell \in \NN$,
\begin{equation}\label{eq:upper bd}
3d_{Y\cup\H}(v_{0k},v_{0\ell}) + 4 \le 3(\delta_Y + 2M_X) + 4 \le |S(o,v_{0\ell};D)|.
\end{equation}
Let $m=|S(o,v_{01};D)|$ and $S(o,v_{01};D)=\{C_1\preceq C_2\preceq \cdots \preceq C_m\}$. Since \eqref{eq:upper bd} implies $3d_{Y\cup\H}(v_{0k},v_{01})+1 \le m-3$, the path $\gamma_{0k}$ penetrates $C_1$, $C_2$, and $C_3$ for any $k\in\NN$ by Lemma \ref{lem:penetration in triangle}. Let $C_3$ be an $H_\lambda$-coset. For each $k \in \NN$, let $a_{0k-}$ and $a_{0k+}$ be the entrance and exit points of $\gamma_{0k}$ in $C_3$. Since $\gamma_{0k}$ is geodesic in $\Gamma(G,X\sqcup\H)$ for each $k\in\NN$, we have
\[
\hd_{\lambda}(a_{01-},a_{0k-}) \le 3C
\]
by Lemma \ref{lem:bounded by 3C}. Since the relative metric $\hd_{\lambda}$ is locally finite (see Definition \ref{def:hyp emb}), the set $\{a_{0k-} \mid k\in\NN\}$ is finite. Hence, there exist $a_1 \in \{a_{0k-} \mid k\in\NN\}$ and a subsequence $(g_{1'k})_{k=1}^\infty$ of $(g_{0k})_{k=1}^\infty$ such that $a_1 \in\gamma_{1'k}$ for any $k\in\NN$. Note $S(o,a_1;D) \ge 1$ since we have $C_1 \in S(o,a_1;D)$ by Lemma \ref{lem:sep coset inclustion}. By $\lim_{k,\ell\to \infty} (g_{1'k},g_{1'\ell})_o^{Y\cup\H} = \infty$, there exists a subsequence $(g_{1k})_{k=1}^\infty$ of $(g_{1'k})_{k=1}^\infty$ such that
\begin{equation}\label{eq:g1k}
    \inf_{k,\ell \in\NN}(g_{1k},g_{1\ell})_o^{Y\cup\H} \ge 2(3(\delta_Y + 2M_X) + |S(o,a_1;D)|+4) + 1 + M_X.
\end{equation}
Note $a_1 \in \gamma_{1k}$ for any $k\in\NN$. By the same argument as $(g_{0k})_{k=1}^\infty$, we can see that there exist $a_2 \in G$ and a subsequence $(g_{2'k})_{k=1}^\infty$ of $(g_{1k})_{k=1}^\infty$ such that $a_2 \in \gamma_{2'k}$ for any $k\in\NN$ and $|S(o,a_2;D)| \ge |S(o,a_1;D)| +1$. The latter inequality comes from the term $|S(o,a_1;D)|+4$ in \eqref{eq:g1k}, which corresponds to $4$ in \eqref{eq:g0k}. By repeating this argument, we can see that there exist $a_1,a_2,\cdots \in G$ and a sequence of subsequences $(g_{1k})_{k=1}^\infty \supset (g_{2k})_{k=1}^\infty \supset \cdots$ satisfying (i) and (ii) below for any $n\in \NN$.
\begin{itemize}
    \item[(i)]
    $\{a_1,\cdots,a_n\}\subset \gamma_{nk}$ for any $k\in\NN$.
    \item[(ii)]
    $|S(o,a_{n+1};D)| \ge |S(o,a_n;D)|+1$.
\end{itemize}
Take the diagonal sequence $(g_{kk})_{k=1}^\infty$, then for any $n\in\NN$, $(g_{kk})_{k=n}^\infty$ is a subsequence of $(g_{nk})_{k=1}^\infty$. Hence, $\{a_1,\cdots,a_n\}\subset \gamma_{nn}$ for any $n\in \NN$. Note that (ii) and Lemma \ref{lem:inclusion} imply $d_{X\cup\H}(o,a_n) < d_{X\cup\H}(o,a_{n+1})$ for any $n\in\NN$. Define the path $\gamma \colon [0,\infty) \to \Gamma(G,X\sqcup\H)$ by
\[
\gamma=\bigcup_{n=1}^\infty \gamma_{nn[a_{n-1},a_n]},
\]
where we define $a_0$ by $a_0=o$ for convenience. the path $\gamma$ is a geodesic ray since for any $n \in \NN$, we have
\[
\sum_{i=1}^n |\gamma_{ii[a_{i-1},a_i]}|
= \sum_{i=1}^n d_{X\cup\H}(a_{i-1},a_i)
= \sum_{i=1}^n |\gamma_{nn[a_{i-1},a_i]}|
= |\gamma_{nn[o,a_n]}|
=d_{X\cup\H}(o,a_n).
\]
Since we have $\lim_{n\to\infty} |S(o,a_n;D)| = \infty$ by (ii) and $(a_n)_{n=1}^\infty$ is a subsequence of $\gamma$, the path $\gamma$ converges to infinity in $\Gamma(G,Y\sqcup\H)$ by Lemma \ref{lem:seq to infinity} and we have $Y\-\lim \gamma = Y\-\lim_{n\to\infty} a_n$.

Finally, we will show $Y\-\lim \gamma = \xi \in \partial\Gamma(G,Y\sqcup\H)$. Since we have $\xi= Y\-\lim_{n\to\infty} g_n =Y\-\lim_{n\to\infty} g_{nn}$, it's enough to show $Y\-\lim_{n\to\infty} a_n=Y\-\lim_{n\to\infty} g_{nn} \in \partial\Gamma(G,Y\sqcup\H)$. For any $R\in\NN$, there exists $N\in\NN$ such that $d_{Y\cup\H}(o,a_N)\ge R+3M_X+2\delta_Y$ since we have $\lim_{n\to\infty} d_{Y\cup\H}(o,a_n)=\infty$ by Lemma \ref{lem:seq to infinity}. Hence, for any $m,n\ge N$, we have $(a_m,g_{nn})_o^{Y\cup\H} \ge R$ by Lemma \ref{lem:estimate gromov prod in Y} applied to $x=a_m, y=g_{nn}, v=a_N, w=a_N$. This implies $\lim_{m,n\to\infty} (a_m,g_{nn})_o^{Y\cup\H} = \infty$, hence $Y\-\lim_{n\to\infty} a_n=Y\-\lim_{n\to\infty} g_{nn}$ in $\partial\Gamma(G,Y\sqcup\H)$.
\end{proof}

We next show that as the limit points of geodesic rays in $\Gamma(G,X\sqcup\H)$ get closer to one another in $\partial\Gamma(G,Y\sqcup\H)$, the geodesic rays have more common separating cosets. Proposition \ref{prop:for cts} has two important corollaries, Corollary \ref{cor:separating coset of ray} and Corollary \ref{cor:for hyperfinite}.

\begin{prop} \label{prop:for cts}
    Let $o\in G$, $\xi \in \partial\Gamma(G,Y\sqcup\H)$ and suppose that $\alpha$ is a geodesic ray in $\Gamma(G,X\sqcup\H)$ from $o$ such that $Y\-\lim \alpha = \xi \in \partial\Gamma(G,Y\sqcup\H)$. Let $S(\alpha,D)=\{C_1\preceq C_2 \preceq \cdots\}$. Then, for any $N\in\NN$, there exists an open neighborhood $U$ of $\xi$ in $\partial\Gamma(G,Y\sqcup\H)$ such that any geodesic ray $\beta$ in $\Gamma(G,X\sqcup\H)$ from $o$ satisfying $Y\-\lim \beta \in U$ penetrates $C_n$ for any $n\in \{1,\cdots,N\}$.
\end{prop}

\begin{proof}
    Given $N\in\NN$, define $R\in\NN$ by
    \begin{equation}\label{eq:R in cts}
            R=2(3(\delta_Y+2M_X)+N+2)+1+M_X+2\delta_Y.
    \end{equation}
    Let $\eta \in \partial\Gamma(G,Y\sqcup\H)$ satisfy $(\xi,\eta)_o^{Y\cup\H}> R$ (see \eqref{eq:gromov prod for boundary}) and $\beta$ be a geodesic ray in $\Gamma(G,X\sqcup\H)$ from $o$ such that $Y\-\lim \beta =\eta$. Let $\alpha=(x_0,x_1,\cdots)$ and $\beta=(y_0,y_1,\cdots)$. It's not difficult to see from \eqref{eq:gromov prod for boundary} that $(\xi,\eta)_o^{Y\cup\H}> R$ implies
    \[
    \liminf_{i,j\to\infty} (x_i,y_j)_o^{Y\cup\H} > R-2\delta_Y.
    \]
    Hence, there exists $k\in\NN$ such that $(x_k,y_k)_o^{Y\cup\H} > R-2\delta_Y$. By Lemma \ref{lem:gromov prod in Y} applied to $o,x_k,y_k,\alpha_{[o,x_k]},\beta_{[o,y_k]}$, there exist $\ell,m$ with $\ell,m \le k$ such that $d_{Y\cup\H} (x_\ell,y_m) \le \delta_Y+2M_X$ and $d_{Y\cup\H} (o,x_\ell) \ge R-2\delta_Y-M_X$. This implies, together with \eqref{eq:qi} in Theorem \ref{thm:osin} and \eqref{eq:R in cts},
    \begin{equation} \label{eq:for xell}
        3d_{Y\cup\H} (x_\ell,y_m)+N+2
        \le 3(\delta_Y+2M_X)+N+2
        \le |S(o,x_\ell;D)|.
    \end{equation}
    Let $L=|S(o,x_\ell;D)|$ and $S(o,x_\ell;D)=\{B_1\preceq\cdots\preceq B_L\}$. By Lemma \ref{lem:penetration in triangle}, $\beta_{[y_0,y_m]}$ penetrates $B_n$ for any $n\in \{1,\cdots,N\}$ since \eqref{eq:for xell} implies $3d_{Y\cup\H} (x_\ell,y_m)+1\le L-N$. Meanwhile, suppose $d_{X\cup\H}(o,x_\ell) \le d_{X\cup\H}(o,\alpha_{out}(C_{N+1}))$ for contradiction, then we have $\alpha_{[o,x_\ell]} \subset \alpha_{[o,\alpha_{out}(C_{N+1})]}$. This implies $|S(o,x_\ell;D)| \le |S(o,\alpha_{out}(C_{N+1});D)| \le N+1$ by Lemma \ref{lem:inclusion}. This contradicts that we get $N+2\le|S(o,x_\ell;D)|$ by \eqref{eq:for xell}. Thus, we have $d_{X\cup\H}(o,\alpha_{out}(C_{N+1})) < d_{X\cup\H}(o,x_\ell)$. Hence, we have $B_n=C_n$ for any $n\in\{1,\cdots,N\}$ by Lemma \ref{lem: sep coset ordered}. Thus, $\beta$ penetrates $C_n$ for any $n\in \{1,\cdots,N\}$. Finally, since the set $V=\{\eta\in \partial\Gamma(G,Y\sqcup\H) \mid (\xi,\eta)_o^{Y\cup\H}> R \}$ is a neighborhood of $\xi$, there exists an open set $U$ of $\partial\Gamma(G,Y\sqcup\H)$ such that $\xi\in U$ and $U\subset V$. This $U$ satisfies the statement for $N$.
\end{proof}

Corollary \ref{cor:separating coset of ray} is the final step to extend the notion of Hull-Osin's separating cosets, which is done in Definition \ref{def:sep coset for vertex and boundry points}.

\begin{cor}\label{cor:separating coset of ray}
    Let $o\in G$ and suppose that $\alpha,\beta$ are geodesic rays in $\Gamma(G,X\sqcup\H)$ from $o$ such that $Y\-\lim \alpha = Y\-\lim \beta \in \partial\Gamma(G,Y\sqcup\H)$. Then, we have $S(\alpha;D)=S(\beta;D)$.
\end{cor}

\begin{proof}
    Let $\alpha=(x_0,x_1,\cdots)$, $\beta=(y_0,y_1,\cdots)$, and $S(\alpha;D)=\{C_1 \preceq C_2\preceq\cdots\}$. By Proposition \ref{prop:for cts}, $\beta$ penetrates $C_n$ for any $n\in\NN$ since $Y\-\lim \beta$ is obviously contained in any open neighborhood of $Y\-\lim \alpha$. For any $n\in\NN$, there exists $k\in \NN$ such that $\{C_1, \cdots,C_{n+1}\} \subset S(o,x_k;D)$. Let $\ell\in\NN$ satisfy $d_{X\cup\H}(o,\beta_{out}(C_{n+1}))\le d_{X\cup\H}(o,y_\ell)$. By Lemma \ref{lem:sep coset inclustion}, we have $\{C_1, \cdots,C_n\} \subset S(o,y_\ell;D)\subset S(\beta;D)$. This implies $S(\alpha;D) \subset S(\beta;D)$. Similarly, we can also see $S(\beta;D) \subset S(\alpha;D)$.
\end{proof}

\begin{defn}\label{def:sep coset for vertex and boundry points}
    Given $x\in G$ and $\xi \in \partial\Gamma(G,Y\sqcup\H)$, we take a geodesic ray $\gamma$ in $\Gamma(G,X\sqcup\H)$ satisfying $\gamma_-=x$ and $Y\-\lim \gamma=\xi$ and define the set of cosets $S(x,\xi;D)$ by
    \[
    S(x,\xi;D)=S(\gamma;D).
    \]
    We call an element of $S(x,\xi;D)$ a $(x,\xi;D)$-\emph{separating coset}.
\end{defn}

\begin{rem}
    By Proposition \ref{prop:existence} and Corollary \ref{cor:separating coset of ray}, Definition \ref{def:sep coset for vertex and boundry points} is well-defined, that is, $\gamma$ above exists and $S(x,\xi;D)$ doesn't depend on $\gamma$. Also, $S(x,\xi;D)$ is exactly the set of all cosets that are essentially penetrated by some geodesic ray $\gamma$ in $\Gamma(G,X\sqcup\H)$ from $x$ with $Y\-\lim \gamma=\xi$.
\end{rem}

We will next show in Corollary \ref{cor:for hyperfinite} that two geodesic rays in $\Gamma(G,X\sqcup\H)$ having the same limit in $\partial\Gamma(G,Y\sqcup\H)$ penetrate the same separating cosets after going far enough. This will play an important role to show hyperfiniteness in the proof of Theorem \ref{thm:main}. We begin with an auxiliary lemma. Lemma \ref{lem:well-known} is a well-known fact, but we write down a sketch of the proof for completeness.

\begin{lem}\label{lem:well-known}
    Suppose that $\Gamma$ is a connected graph, $x$ is a vertex, and $\beta=(y_0,y_1,\cdots)$ is a geodesic ray. Then, there exists $k\in\NN$ such that for any geodesic path $p$ from $x$ to $y_k$, the path $p\beta_{[y_k,\infty)}$ is a geodesic ray.
\end{lem}

\begin{proof}
    Define the map $f \colon \NN\to\ZZ$ by $f(n)=d_\Gamma(y_n,y_0)-d_\Gamma(y_n,x)$. Since $\beta$ is geodesic, we have $f(n)=n-d_\Gamma(y_n,x), \forall n\in\NN$. This and $|d_\Gamma(y_{n+1},x)-d_\Gamma(y_n,x)|\le 1, \forall n\in\NN$ imply that $f$ is non-decreasing. The map $f$ is also bounded above by $d_\Gamma(y_0,x)$. Hence, there exists $k\in\NN$ such that $f(n)=f(k)$ for any $n$ with $n\ge k$. Since this implies $d_\Gamma(y_n,x)=d_\Gamma(y_n,y_k)+d_\Gamma(y_k,x)$ for any $n$ with $n\ge k$, this $k$ satisfies the statement.
\end{proof}

\begin{figure}[htbp]
  \begin{center}
 \hspace{0mm} 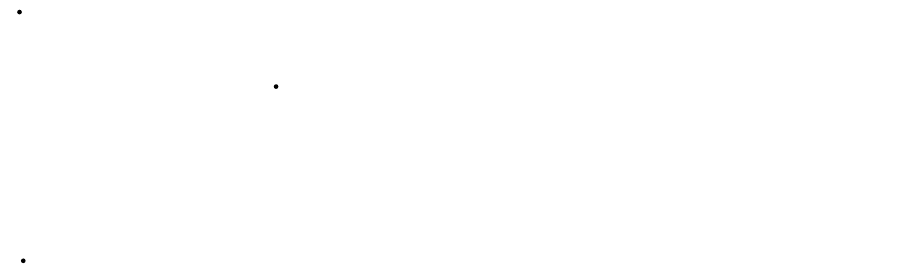
  \end{center}
   \vspace{-3mm}
  \caption{Proof of Corollary \ref{cor:for hyperfinite}}
  \label{Cor3_28}
\end{figure}

\begin{cor} \label{cor:for hyperfinite}
    Suppose that $\alpha,\beta$ are geodesic rays in $\Gamma(G,X\sqcup\H)$ that converge to infinity in $\Gamma(G,Y\sqcup\H)$ and let $S(\alpha;D)=\{C_1\preceq C_2\preceq\cdots\}$. If $Y\-\lim\alpha=Y\-\lim\beta$ in $\partial\Gamma(G,Y\sqcup\H)$, then there exists $N\in\NN$ such that for any $n\ge N$, $\beta$ penetrates $C_n$ satisfying $d_{X\cup\H}(\beta_-,C_n)<d_{X\cup\H}(\beta_-,C_{n+1})$ and we have, when $C_n$ is an $H_\lambda$-coset,
    \begin{align} \label{eq:for hyperfinite}
        \hd_{\lambda}(\alpha_{in}(C_n), \beta_{in}(C_n)) \le 4C
        {\rm ~~~and~~~}
        \hd_{\lambda}(\alpha_{out}(C_n), \beta_{out}(C_n)) \le 4C.
    \end{align}
\end{cor}

\begin{proof}
    Let $\beta=(y_0,y_1,\cdots)$. By applying Lemma \ref{lem:well-known} to $\Gamma(G,X\sqcup\H)$, there exist $k\in\NN$ and a geodesic path $p$ in $\Gamma(G,X\sqcup\H)$ from $\alpha_-$ to $y_k$ such that the path $\gamma$ defined by $\gamma=p\beta_{[y_k,\infty)}$ is a geodesic ray in $\Gamma(G,X\sqcup\H)$. By $\gamma_{[y_k,\infty)}=\beta_{[y_k,\infty)}$, we have $Y\-\lim \gamma=Y\-\lim \beta=Y\-\lim\alpha$. This implies $S(\gamma;D)=S(\alpha;D)$ by Corollary \ref{cor:separating coset of ray}. Hence, $\gamma$ penetrates $C_n$ for any $n\in\NN$ by Lemma \ref{lem:geod ray penetrate}. Note that $|S(\gamma;D)|=\infty$ implies $\lim_{n\to\infty} d_{X\cup\H}(\gamma_-, \gamma_{in}(C_n))=\infty$. Hence, there exists $N\ge 2$ such that $d_{X\cup\H}(\gamma_-, \gamma_{in}(C_{N-1}))>d_{X\cup\H}(\gamma_-, y_k)$. By $\gamma_{[y_k,\infty)}=\beta_{[y_k,\infty)}$, the path $\beta$ penetrates $C_n$ for any $n\ge N-1$. Let $n\ge N$ and $C_n$ be an $H_\lambda$-coset. By $C_{n-1}\preceq C_n$ and $\gamma_-=\alpha_-$, we have $d_{X\cup\H}(\gamma_-, C_{n-1})<d_{X\cup\H}(\gamma_-, C_n)$. This implies
    \begin{align*}
        d_{X\cup\H}(\beta_-, C_n)-d_{X\cup\H}(\beta_-, C_{n-1})
        &=d_{X\cup\H}(y_k, \beta_{in}(C_n))-d_{X\cup\H}(y_k, \beta_{in}(C_{n-1})) \\
        &=d_{X\cup\H}(y_k, \gamma_{in}(C_n))-d_{X\cup\H}(y_k, \gamma_{in}(C_{n-1})) \\
        &=d_{X\cup\H}(\gamma_-, C_n)-d_{X\cup\H}(\gamma_-, C_{n-1})
        >0.
    \end{align*}
    Hence, we have $\hd_{\lambda}(\alpha_{in}(C_n), \beta_{in}(C_n)) \le 4C$ by applying Lemma \ref{lem:bounded by 4C} to $C_{n-1},C_n,\alpha,\beta$. Similarly, we have $\hd_{\lambda}(\alpha_{out}(C_n), \beta_{out}(C_n)) \le 4C$ by applying Lemma \ref{lem:bounded by 4C} to $C_n,C_{n+1},\alpha,\beta$.
\end{proof}

Finally, we show that if the limit points in $\partial\Gamma(G,X\sqcup\H)$ of geodesic rays in $\Gamma(G,X\sqcup\H)$ are convergent, then their limit points in $\partial\Gamma(G,Y\sqcup\H)$ are also convergent. Proposition \ref{prop:cts from X to Y} can be considered as opposite to Proposition \ref{prop:for cts}. This will become clear in Proposition \ref{prop:homeo}.

\begin{prop}\label{prop:cts from X to Y}
Let $o\in G$ and suppose that $\alpha$ is a geodesic ray in $\Gamma(G,X\sqcup\H)$ from $o$ converging to infinity in $\Gamma(G,Y\sqcup\H)$. For any open neighborhood $U$ of $Y\-\lim \alpha$ in $\partial\Gamma(G,Y\sqcup\H)$, there exists an open neighborhood $V$ of $X\-\lim \alpha$ in $\partial\Gamma(G,X\sqcup\H)$ such that if a geodesic ray $\beta$ in $\Gamma(G,X\sqcup\H)$ from $o$ converging to infinity in $\Gamma(G,Y\sqcup\H)$ satisfies $X\-\lim \beta \in V$, then we have $Y\-\lim \beta \in U$.
\end{prop}

\begin{proof}
    Let $\alpha=(x_0,x_1,\cdots)$. For any open neighborhood $U$ of $Y\-\lim \alpha$ in $\partial\Gamma(G,Y\sqcup\H)$, there exists $R\in\NN$ such that 
    \begin{equation} \label{eq:open U}
        \{ \eta \in \partial\Gamma(G,Y\sqcup\H) \mid (Y\-\lim \alpha, \eta)_o^{Y\cup\H}\ge R \}\subset U
    \end{equation}
    by Proposition \ref{prop:gromov topo}. We define $R_1=R+\delta_X+3M_X+2\delta_Y$. Since we have $\lim_{n\to\infty}d_{Y\cup\H}(o,x_n)=\infty$ by Lemma \ref{lem:seq to infinity}, there exists $N\in\NN$ such that
    \[
    d_{Y\cup\H}(o,x_{N}) \ge R_1.
    \]
    By Proposition \ref{prop:gromov topo}, there exists an open neighborhood $V$ of $X\-\lim \alpha$ in $\partial\Gamma(G,X\sqcup\H)$ such that
    \begin{equation}\label{eq:open V}
        V \subset \{ \eta' \in \partial\Gamma(G,X\sqcup\H) \mid (X\-\lim \alpha, \eta')_o^{X\cup\H}>N+2\delta_X \}.
    \end{equation}
    We show that $V$ satisfies the statement. Let $\beta=(y_0,y_1,\cdots)$ be a geodesic ray in $\Gamma(G,X\sqcup\H)$ from $o$ converging to infinity in $\Gamma(G,Y\sqcup\H)$ such that $X\-\lim \beta \in V$. By \eqref{eq:open V}, it's not difficult to see $\liminf_{i,j\to\infty}(x_i,y_j)_o^{X\cup\H}>N$. Since $\alpha$ and $\beta$ are geodesic rays in $\Gamma(G,X\sqcup\H)$ from $o$, this implies
    \[
    d_{Y\cup\H}(x_N,y_N) \le d_{X\cup\H}(x_N,y_N) \le \delta_X.
    \]
    Hence, for any $n,m\ge N$, we have $(x_{n},y_{m})_o^{Y\cup\H} \ge R_1-(\delta_X+3M_X)-2\delta_Y=R$ by applying Lemma \ref{lem:estimate gromov prod in Y} to $\alpha_{[o,x_n]}, \beta_{[o,y_m]},x=x_n,y=y_m,v=x_N,w=y_N$. This implies
    \[
    (Y\-\lim \alpha,Y\-\lim \beta)_o^{Y\cup\H} \ge \liminf_{i,j\to\infty} (x_i,y_j)_o^{Y\cup\H} \ge R.
    \]
    Hence, we have $Y\-\lim \beta \in U$ by \eqref{eq:open U}.
\end{proof}

\section{Proof of main theorem} \label{sec:main thm}
In this section, we prove Theorem \ref{thm:main}. We follow the approach of \cite{NV}, where they gave another proof of the fact that the action of any hyperbolic group on its Gromov boundary is hyperfinite. Their approach goes as follows. Given a hyperbolic group $G$ with a finite symmetric generating set $S$, we fix a well-order on $S$. This order induces the lexicographic order $\lelex$ on $S^\NN$, which enables us to pick for $\xi\in\partial G$, a geodesic ray from $1$ to $\xi$ in the Cayley graph $\Gamma(G,S)$ whose label is the minimum in $(S^\NN,\lelex)$ among all such geodesic rays. This defines an injective Borel measurable map from $\partial G$ to $S^\NN$ that Borel reduces a finite index subequivalence relation of the orbit equivalence relation $E_G^{\partial G}$ to the tail equivalence relation $E_t(S)$, thereby hyperfiniteness of $E_t(S)$ implies hyperfiniteness of $E_G^{\partial G}$.

We first verify in Lemma \ref{lem:Gromov boundary is Polish} that the Gromov boundary is a Polish space (see Definition \ref{def:Polish}). Actually, we can show slightly more. It's interesting to know whether the statement of Lemma \ref{lem:Gromov boundary is Polish} holds for any completely metrizable hyperbolic space, that is, whether $S\cup\partial S$ is completely metrizable for any completely metrizable hyperbolic space $S$. For a graph $\Gamma$, we denote its vertex set by $V(\Gamma)$. The proof below uses that $V(\Gamma)$ is discrete.

\begin{lem}\label{lem:Gromov boundary is Polish}
For any hyperbolic graph $\Gamma$, the topological space $V(\Gamma)\cup \partial \Gamma$ is completely metrizable. If in addition $V(\Gamma)$ is countable, then $V(\Gamma)\cup \partial \Gamma$ is Polish.
\end{lem}

\begin{proof}
Fix a vertex $o\in \Gamma$ and take the map $D \colon (\Gamma\cup\partial\Gamma)^2 \to [0,\infty)$ and constants $\e,\e'>0$ as in Proposition \ref{prop:visual metric}. We define the map $\tD \colon (V(\Gamma)\cup\partial\Gamma)^2 \to [0,\infty)$ by $\tD(x,y)=D(x,y)$ if $x\neq y$ and $\tD(x,y)=0$ if $x=y$. By Proposition \ref{prop:visual metric}, it's not difficult to see that $\tD$ is a metric and the metric topology of $\tD$ coincides with the relative topology of $\mathcal{O}_\Gamma$ on $V(\Gamma)\cup \partial \Gamma$ (see Proposition \ref{prop:gromov topo}). Here, we used discreteness of $(V(\Gamma), \mathcal{O}_\Gamma|_{V(\Gamma)})$ since the metric topology of $\tD$ on $V(\Gamma)$ is discrete by Remark \ref{rem:D not metric}. We will show that $\tD$ is complete. Let $(x_n)_{n=1}^\infty$ be a Cauchy sequence of $(V(\Gamma)\cup\partial\Gamma,\tD)$. Since $V(\Gamma)$ is dense in $V(\Gamma)\cup\partial\Gamma$, we can take for each $n\in \NN$, a vertex $y_n\in V(\Gamma)$ such that $\tD(x_n,y_n)\le \frac{1}{n}$. The sequence $(y_n)_{n=1}^\infty$ in $V(\Gamma)$ is also a Cauchy sequence in the metric $\tD$ and it's enough to show that $(y_n)_{n=1}^\infty$ is convergent. If there exists a constant subsequence $(y_{n_k})_{k=1}^\infty$ (i.e. $y_{n_k}\equiv y$ for some $y\in V(\Gamma)$), then $(y_n)_{n=1}^\infty$ converges to $y$ in $\tD$. If there is no constant subsequence of $(y_n)_{n=1}^\infty$, then there exists a subsequence $(y_{n_k})_{k=1}^\infty$ whose elements are all distinct. This implies $\tD(y_{n_k},y_{n_\ell})=D(y_{n_k},y_{n_\ell})$ for any $k\neq\ell$. Since this implies $\lim_{k,\ell\to\infty} (y_{n_k},y_{n_\ell})_o^\Gamma=\infty$, the sequence $(y_{n_k})_{k=1}^\infty$ converges to some $y \in \partial\Gamma$ in $\mathcal{O}_\Gamma$. Hence, $(y_n)_{n=1}^\infty$ converges to $y$ as well. Thus, $V(\Gamma)\cup \partial \Gamma$ is completely metrizable with $\tD$. If $V(\Gamma)$ is countable, $V(\Gamma)$ is a countable dense subset of $V(\Gamma)\cup \partial \Gamma$. Hence, $V(\Gamma)\cup \partial \Gamma$ is Polish.
\end{proof}

From here up to Lemma \ref{lem:Phi cts}, suppose that $G$ is a countable group, $X$ is a subset of $G$, and $\{H_\lambda\}_{\lambda\in\Lambda}$ is a countable collection of subgroups of $G$ hyperbolically embedded in $(G,X)$. Let $C>0$ as in Proposition \ref{prop:C} and fix $D>0$ satisfying $D\ge 3C$ as in \eqref{eq:D>3C}. We also define the subset $Y$ of $G$ by $Y=\{y \in G \mid S(1,y;D)= \emptyset\}$ as in Theorem \ref{thm:osin}. The difference from Section \ref{sec:Gromov} is that we assume $G$ and $\Lambda$ are countable.

Since $X\sqcup\H$ is countable, we fix a well-order $\le$ on $X\sqcup\H$ by some injection $X\sqcup\H \inj \NN$. The lexicographic order $\lelex$ on $(X\sqcup\H)^\NN$ is naturally defined from the order $\le$ on $X\sqcup\H$, that is, for $w_0=(s_1,s_2,\cdots), w_1=(t_1,t_2,\cdots) \in (X\sqcup\H)^\NN$, we have
\[
w_0 <_{\rm lex} w_1 \iff \exists n\in\NN {\rm ~s.t.~} \left(\bigwedge_{i<n} s_i=t_i \right) \wedge s_n<t_n.
\]
Similarly, we define the lexicographic order $\lelex$ on $(X\sqcup\H)^n$ for each $n\in\NN$.
Note that $(X\sqcup\H)^\NN$ becomes a Polish space with the product topology as mentioned in Section \ref{subsec:descriptiveset theory}.

As suggested at the beginning of this section, the important step of the proof of Theorem \ref{thm:main} is for a boundary point $\xi\in \partial\Gamma(G,Y\sqcup\H)$, picking one geodesic ray in $\Gamma(G,X\sqcup\H)$ from $1$ to $\xi$ in a Borel way. This is done by reading the labels of all geodesic rays from $1$ to $\xi$ and comparing these labels by the lexicographic order defined above. Definition \ref{def:gamma w} and Definition \ref{def:G(xi)} are for setting up notations for the labeling.

\begin{defn}\label{def:gamma w}
    For $w=(s_1,s_2,\cdots)\in (X\sqcup\H)^\NN$, we define the infinite path $\gamma^w$ from $1\in G$ in $\Gamma(G,X\sqcup\H)$ by $\gamma^w=(1,x_{w,1},x_{w,2},\cdots)$, where the $n$-th vertex $x_{w,n}$ is defined by
    \[
    x_{w,n}=s_1\cdots s_n
    \]
    for each $n\in\NN$.
\end{defn}

\begin{defn}\label{def:G(xi)}
    For $\xi \in \partial\Gamma(G,Y\sqcup\H)$, we define the subset $\G(\xi)$ of $(X\sqcup\H)^\NN$ by
    \[
    \G(\xi)=\{w\in (X\sqcup\H)^\NN \mid \gamma^w {\rm ~is~geodesic~in~}\Gamma(G,X\sqcup\H) {\rm ~and~} Y\-\lim \gamma^w=\xi \in \partial\Gamma(G,Y\sqcup\H)\}.
    \]
\end{defn}

\begin{rem}
    Note that the condition $Y\-\lim \gamma^w=\xi$ in Definition \ref{def:G(xi)} implicitly requires that $\gamma^w$ converges to infinity in $\Gamma(G,Y\sqcup\H)$ (see Definition \ref{def:Y-lim}). Also, for any $\xi \in \partial\Gamma(G,Y\sqcup\H)$, the set $\G(\xi)$ is nonempty by Proposition \ref{prop:existence}.
\end{rem}

We will show that picking one geodesic ray from $\G(\xi)$ is possible in Corollary \ref{cor:minimum exists}. Lemma \ref{lem:closed} is the auxiliary lemma for this.

\begin{lem}\label{lem:closed}
    For any $\xi \in \partial\Gamma(G,Y\sqcup\H)$, the set $\G(\xi)$ is closed in $(X\sqcup\H)^\NN$.
\end{lem}

\begin{proof}
    Note that $(X\sqcup\H)^\NN$ is metrizable. Suppose that a sequence $(w_i)_{i=1}^\infty$ in $\G(\xi)$ converges to $w\in (X\sqcup\H)^\NN$. It's straightforward to see that $\gamma^w$ is geodesic in $\Gamma(G,X\sqcup\H)$. We will first show $|S(\gamma^w;D)|=\infty$. Let $S(1,\xi;D)=\{C_1\preceq C_2 \preceq \cdots\}$ (see Definition \ref{def:sep coset for vertex and boundry points}). Since $w_i\in \G(\xi)$ implies $S(\gamma^{w_i};D)=S(1,\xi;D)$ for any $i\in\NN$, we have for any $i,n\in\NN$,
    \begin{align}\label{eq:out}
        d_{X\cup\H}(1,\gamma^{w_i}_{out}(C_n))
        =d_{X\cup\H}(1,\gamma^{w_i}_{in}(C_n))+1
        =d_{X\cup\H}(1,C_n)+1
    \end{align}
    by Lemma \ref{lem:distance of S gamma}. For any $N\in \NN$, we define $k$ by $k=d_{X\cup\H}(1,C_{N+1})+1$. Since $(w_i)_{i=1}^\infty$ converges to $w$, there exists $I\in\NN$ such that $x_{w_I,m}=x_{w,m}$ for any $m\in\{1,\cdots,k\}$ (see Definition \ref{def:gamma w}). This implies
    $
    \{C_1,\cdots,C_N\} \subset S(1,x_{w_I,k};D)=S(1,x_{w,k};D) \subset S(\gamma^w;D)
    $
    by \eqref{eq:out} and Lemma \ref{lem: sep coset ordered}. Hence, we have $|S(\gamma^w;D)|=\infty$ since $N\in\NN$ is arbitrary. Since $(w_i)_{i=1}\infty$ converges to $w$, $(X\-\lim\gamma^{w_i})_{i=1}\infty$ converges to $X\-\lim\gamma^w$ in $\partial\Gamma(G,X\sqcup\H)$. Hence, $(Y\-\lim\gamma^{w_i})_{i=1}\infty$ converges to $Y\-\lim\gamma^w$ in $\partial\Gamma(G,Y\sqcup\H)$ by Proposition \ref{prop:cts from X to Y}. This implies $Y\-\lim\gamma^w=\xi$ by $Y\-\lim\gamma^{w_i}=\xi, \forall i\in\NN$. Thus, we have $w\in \G(\xi)$.
\end{proof}

\begin{cor}\label{cor:minimum exists}
    For any $\xi \in \partial\Gamma(G,Y\sqcup\H)$, the element $\min_{\lelex} \G(\xi)$ exists.
\end{cor}

\begin{proof}
    For each $n\in \NN$, we define the element $s_{\xi,n}\in X\sqcup\H$ and the subset $\G(\xi)_n$ of $\G(\xi)$ inductively as follows:
    \begin{align*}
    s_{\xi,1}&=\min \{ s_1 \in (X\sqcup\H,\le) \mid \exists w\in \G(\xi) {\rm ~s.t.~} w=(s_1,s_2,\cdots) \}, \\
    \G(\xi)_1&=\{ w\in \G(\xi) \mid w=(s_{\xi,1},s_2,\cdots) \}, \\
    s_{\xi,n+1}&=\min \{ s_{n+1} \in (X\sqcup\H,\le) \mid \exists w\in \G(\xi)_n {\rm ~s.t.~} w=(s_{\xi,1},\cdots,s_{\xi,n},s_{n+1},\cdots) \}, \\
    \G(\xi)_{n+1}&=\{ w\in \G(\xi)_n \mid w=(s_{\xi,1},\cdots,s_{\xi,n},s_{\xi,n+1},\cdots) \}.
    \end{align*}
    Note that each $\G(\xi)_n$ is nonempty since $\G(\xi)$ is nonempty. We define the element $w_\xi\in (X\sqcup\H)^\NN$ by $w_\xi=(s_{\xi,1},s_{\xi,2},s_{\xi,3},\cdots)$ and take an element $w_n \in \G(\xi)_n$ for each $n\in\NN$. Since $(w_n)_{n=1}\infty$ converges to $w_\xi$ in $(X\sqcup\H)^\NN$, we have $w_\xi\in \G(\xi)$ by Lemma \ref{lem:closed}. Since $w_\xi \in\bigcap_{n=1}^\infty \G(\xi)_n$, we have $w_\xi \lelex w$ for any $w\in \G(\xi)$.
\end{proof}

Definition \ref{def:Phi} below is well-defined by Corollary \ref{cor:minimum exists}.

\begin{defn} \label{def:Phi}
    We define the map $\Phi \colon \partial\Gamma(G,Y\sqcup\H) \to (X\sqcup\H)^\NN$ by
    \[
    \Phi(\xi) = \min\nolimits_{\lelex} \G(\xi).
    \]
    For each $\xi \in \partial\Gamma(G,Y\sqcup\H)$, we denote $\Phi(\xi)=(s_{\xi,1},s_{\xi,2},s_{\xi,3},\cdots)$.
\end{defn}

We will show that the map $\Phi(\xi)$ is injective and continuous in Lemma \ref{lem:Phi inj} and Lemma \ref{lem:Phi cts}. This will finish the step of picking a geodesic ray for a boundary point.

\begin{lem}\label{lem:Phi inj}
    The map $\Phi \colon \partial\Gamma(G,Y\sqcup\H) \to (X\sqcup\H)^\NN $ is injective.
\end{lem}

\begin{proof}
    This follows since we have $\xi=Y\-\lim \gamma^{\Phi(\xi)}$ for any $\xi \in \partial\Gamma(G,Y\sqcup\H)$.
\end{proof}

Recall that we put the discrete topology on the countable set $X\sqcup\H$ and the product topology on $(X\sqcup\H)^\NN$.

\begin{lem} \label{lem:Phi cts}
    The map $\Phi \colon \partial\Gamma(G,Y\sqcup\H) \to (X\sqcup\H)^\NN$ is continuous.
\end{lem}

\begin{proof}
    It's enough to show that for any $\xi \in \partial\Gamma(G,Y\sqcup\H)$ and any $k\in\NN$, there exists an open neighborhood $U$ of $\xi$ in $\partial\Gamma(G,Y\sqcup\H)$ such that $s_{\eta,n}=s_{\xi,n}$ for any $\eta \in U$ and any $n\in \{1,\cdots,k\}$. Indeed, given $\xi\in \partial\Gamma(G,Y\sqcup\H)$, the open sets $\{V_k\}_{k=1}^\infty$ defined by $V_k=\{(s_n)_{n=1}^\infty \in (X\sqcup\H)^\NN \mid s_n=s_{\xi,n}, \forall n\in \{1,\cdots,k\} \}$ form a neighborhood basis of $\Phi(\xi)$. Hence, for any open neighborhood $V$ of $\Phi(\xi)$, there exists $k\in \NN$ such that $V_k \subset V$. For this $k$, we will be able to take an open neighborhood $U$ of $\xi$ such that $\Phi(U)\subset V_k$. This will imply continuity of $\Phi$ at $\xi$. Hence, we will get continuity of $\Phi$ since $\xi \in \partial\Gamma(G,Y\sqcup\H)$ is arbitrary.
    
    Let $S(1,\xi;D)=\{C_1\preceq C_2 \preceq\cdots\}$ (see Definition \ref{def:sep coset for vertex and boundry points}). By $\lim_{i\to\infty}d_{X\cup\H}(1,C_i)=\infty$, there exists $N\in\NN$ such that $d_{X\cup\H}(1,C_N)> k$. By Proposition \ref{prop:for cts}, there exists an open neighborhood $U$ of $\xi$ in $\partial\Gamma(G,Y\sqcup\H)$ such that for any $\eta\in U$ and any $w \in \G(\eta)$, the path $\gamma^w$ penetrates $C_i$ for any $i\in\{1,\cdots,N\}$. Let $\eta\in U$. We define $m$ by $m=d_{X\cup\H}(1,C_N)$ for brevity and let $C_N$ be an $H_\lambda$-coset. We claim
    \begin{equation}\label{eq:xi eta}
         \begin{split}
         &\{(s_1,\cdots,s_m) \in (X\sqcup\H)^m \mid \exists w\in\G(\eta) {\rm ~s.t.~} w=(s_1,\cdots,s_m,\cdots) \} \\
        =& \{(s_1,\cdots,s_m) \in (X\sqcup\H)^m \mid \exists w\in\G(\xi) {\rm ~s.t.~} w=(s_1,\cdots,s_m,\cdots) \}.    
         \end{split}
    \end{equation}
    Indeed, for any $w_1\in \G(\eta)$ and $w_2\in \G(\xi)$, let $e_1,e_2$ be the edges in $\Gamma(G,X\sqcup\H)$ whose labels are in $H_\lambda$ such that $e_1$ is from $\gamma^{w_1}_{in}(C_N)$ to $\gamma^{w_2}_{out}(C_N)$ and $e_2$ is from $\gamma^{w_2}_{in}(C_N)$ to $\gamma^{w_1}_{out}(C_N)$. We have $x_{w_1,m}=\gamma^{w_1}_{in}(C_N)$, $x_{w_1,m+1}=\gamma^{w_1}_{out}(C_N)$, $x_{w_2,m}=\gamma^{w_2}_{in}(C_N)$, and $x_{w_2,m+1}=\gamma^{w_2}_{out}(C_N)$ by Lemma \ref{lem:lem4.6 osin}. Hence, the paths $\alpha_1,\alpha_2$ defined by
    \[
    \alpha_1=\gamma^{w_1}_{[1,x_{w_1,m}]}e_1\gamma^{w_2}_{[x_{w_2,m+1},\infty)}
    {\rm ~~and~~}
    \alpha_2=\gamma^{w_2}_{[1,x_{w_2,m}]}e_2\gamma^{w_1}_{[x_{w_1,m+1},\infty)}
    \]
    are geodesic in $\Gamma(G,X\sqcup\H)$. By $Y\-\lim\alpha_1=Y\-\lim\gamma^{w_2}=\xi$ and $Y\-\lim\alpha_2=Y\-\lim\gamma^{w_1}=\eta$, we have $\alpha_1\in \G(\xi)$ and $\alpha_2 \in \G(\eta)$. This implies \eqref{eq:xi eta}. By $\eqref{eq:xi eta}$, we have $s_{\eta,n}=s_{\xi,n}$ for any $n\in \{1,\cdots,m\}$.
\end{proof}

\begin{figure}[htbp]
 \hspace{0mm} 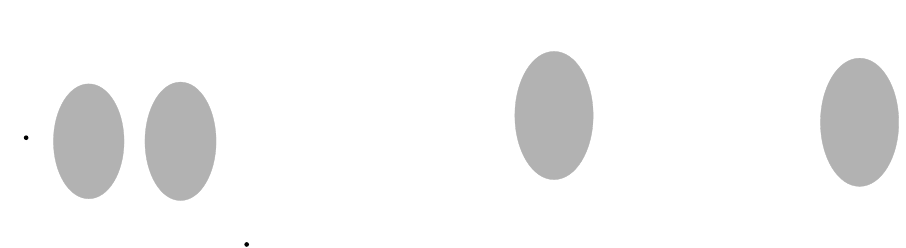
   \vspace{-3mm}
  \caption{Proof of Proposition \ref{prop:hyperfinite}}
  \label{Prop4_10}
\end{figure}

We are now ready to show hyperfiniteness of the boundary action in Proposition \ref{prop:hyperfinite}, which is essentially the proof of Theorem \ref{thm:main}. The difference of the conditions in Proposition \ref{prop:hyperfinite} from those at the beginning of this section is that we further assume that $\Lambda$ is finite.

\begin{prop}\label{prop:hyperfinite}
    Suppose that $G$ is a countable group, $X$ is a subset of $G$, and $\{H_\lambda\}_{\lambda\in\Lambda}$ is a finite collection of subgroups of $G$ hyperbolically embedded in $(G,X)$. Let $C>0$ as in Proposition \ref{prop:C} and fix $D>0$ satisfying $D\ge 3C$. We also define the subset $Y$ of $G$ by $Y=\{y \in G \mid S(1,y;D)= \emptyset\}$ as in \eqref{eq:Y}. Then, the orbit equivalence relation $E_G$ on $\partial\Gamma(G,Y\sqcup\H)$ induced by the action $G\act \partial\Gamma(G,Y\sqcup\H)$ is a hyperfinite CBER.
\end{prop}

\begin{proof}
    Since $G$ is countable, $\partial\Gamma(G,Y\sqcup\H)$ is Polish being a closed subset of the Polish space $G\cup\Gamma(G,Y\sqcup\H)$ by Lemma \ref{lem:Gromov boundary is Polish}. Since $\partial\Gamma(G,Y\sqcup\H)$ is a Polish space and $G$ is a countable group acting on $\partial\Gamma(G,Y\sqcup\H)$ homeomorphically, $E_G$ is a CBER by Lemma \ref{lem:orbit equivalence is cber}. We will show that $E_G$ is hyperfinite. Define the subsets $R,R_1$ of $(\partial\Gamma(G,Y\sqcup\H))^2$ by $R=(\Phi\times\Phi)^{-1}(E_t(X\sqcup\H))$ and $R_1=R\cap E_G$ (see Definition \ref{def:tail equivalence} for $E_t(X\sqcup\H)$). Since $E_t(X\sqcup\H)$ is hyperfinite by Proposition \ref{prop:DJK} and the map $\Phi$ is an injective Borel measurable map by Lemma \ref{lem:Phi inj} and Lemma \ref{lem:Phi cts}, $R$ is a hyperfinite CBER on $\partial\Gamma(G,Y\sqcup\H)$. Hence, $R_1$ is also hyperfinite. We define the constant $K$ by
    \begin{equation}\label{eq:K}
            K=\left( \max_{\lambda\in\Lambda} |\{h\in H_\lambda \mid \hd_{\lambda}(1,h)\le 4C \}| \right)^2.
    \end{equation}
    Note $K<\infty$, since $\Lambda$ is finite by our assumption and each $\hd_{\lambda}$ is locally finite. We will show that each $E_G$-class is composed of at most $K$ equivalence classes of $R_1$. This implies that $E_G$ is hyperfinite by Proposition \ref{prop:JKL}. Suppose for contradiction that there exist $\xi_0,\xi_1,\cdots,\xi_K \in \partial\Gamma(G,Y\sqcup\H)$ such that $(\xi_i,\xi_j)\in E_G \setminus R_1$ for any distinct $i,j\in\{0,1,\cdots,K\}$ (i.e. $i\neq j$). For each $i\in \{0,1,\cdots,K\}$, there exists $g_i \in G$ such that $g_i\xi_i = \xi_0$ by $(\xi_i,\xi_0)\in E_G$. We take $g_0=1$. Let
    $
    S(\gamma^{\Phi(\xi_0)};D)=\{C_1\preceq C_2\preceq\cdots\}.
    $
    By Corollary \ref{cor:for hyperfinite}, there exists $N_0\in\NN$ such that for any $i\in \{0,1,\cdots,K\}$ and any $n\ge N_0$, the path $g_i\gamma^{\Phi(\xi_i)}$ penetrates $C_n$ and satisfies \eqref{eq:for hyperfinite}. Define $m_i$ by
    \[
    m_i= d_{X\cup\H}(g_i,(g_i \gamma^{\Phi(\xi_i)})_{in}(C_{N_0}))
    \]
    for each $i\in\{0,1,\cdots,K\}$. Note $m_0=d_{X\cup\H}(1,C_{N_0})$. For any distinct $i,j\in\{0,1,\cdots,K\}$, $(\xi_i,\xi_j)\notin R_1$ implies $(\Phi(\xi_i),\Phi(\xi_j))\notin E_t(X\sqcup\H)$. Hence, there exists $k\in\NN$ such that
    \begin{align} \label{eq:tail}
        (s_{\xi_i,m_i+1},\cdots,s_{\xi_i,m_i+k}) \neq (s_{\xi_j,m_j+1},\cdots,s_{\xi_j,m_j+k})
    \end{align}
    for any distinct $i,j\in\{0,1,\cdots,K\}$. On the other hand, there exists $N_1\in \NN$ such that $d_{X\cup\H}(1,\gamma_{in}^{\Phi(\xi_0)}(C_{N_1})) > m_0+k$ by $\lim_{n\to\infty}d_{X\cup\H}(1,\gamma_{in}^{\Phi(\xi_0)}(C_n)) = \infty$. Define $\ell\in\NN$ by $\ell=d_{X\cup\H}(C_{N_0},C_{N_1})+1$. By Lemma \ref{lem:distance of cosets}, we have
    \begin{align*}
        d_{X\cup\H}((g_i\gamma^{\Phi(\xi_i)})_{in}(C_{N_0}),(g_i\gamma^{\Phi(\xi_i)})_{in}(C_{N_1}))
        = \ell
        {\rm ~~and~~}
        d_{X\cup\H}(g_i,(g_i\gamma^{\Phi(\xi_i)})_{in}(C_{N_1}))=m_i+\ell
    \end{align*}
    for any $i\in\{0,1,\cdots,K\}$. In particular, $m_0+\ell=d_{X\cup\H}(1,\gamma_{in}^{\Phi(\xi_0)}(C_{N_1})) > m_0+k$ implies $\ell > k$. By \eqref{eq:for hyperfinite}, the set
    \[
    \Big\{\Big((g_i\gamma^{\Phi(\xi_i)})_{in}(C_{N_0}),(g_i\gamma^{\Phi(\xi_i)})_{in}(C_{N_1})\Big) \in G\times G \;\Big|\; i=0,1,\cdots,K \Big\}
    \]
    has at most $K$ elements (see \eqref{eq:K}). Hence, by Pigeonhole principle, there exist distinct $i,j\in\{0,1,\cdots,K\}$ such that 
    \begin{align} \label{eq:endpoints}
    (g_i\gamma^{\Phi(\xi_i)})_{in}(C_{N_0})=(g_j\gamma^{\Phi(\xi_j)})_{in}(C_{N_0})
    {\rm ~~and~~}
    (g_i\gamma^{\Phi(\xi_i)})_{in}(C_{N_1})=(g_j\gamma^{\Phi(\xi_j)})_{in}(C_{N_1}).
    \end{align}
    By \eqref{eq:endpoints}, minimality of $\Phi(\xi_i)$ in $(\G(\xi_i), \lelex)$, and minimality of $\Phi(\xi_j)$ in $(\G(\xi_j), \lelex)$, we have
    \begin{equation}\label{eq:sxii coincides sxij}
        (s_{\xi_i,m_i+1},\cdots,s_{\xi_i,m_i+\ell}) = (s_{\xi_j,m_j+1},\cdots,s_{\xi_j,m_j+\ell}).
    \end{equation}
    Indeed, suppose for contradiction that \eqref{eq:sxii coincides sxij} doesn't hold. We assume without loss of generality, $(s_{\xi_i,m_i+1},\cdots,s_{\xi_i,m_i+\ell}) <_{{\rm lex}} (s_{\xi_j,m_j+1},\cdots,s_{\xi_j,m_j+\ell})$ and define $a,b$ by $a=(g_i\gamma^{\Phi(\xi_i)})_{in}(C_{N_0})$ and $b=(g_i\gamma^{\Phi(\xi_i)})_{in}(C_{N_1})$ for brevity. By \eqref{eq:endpoints}, the path
    \[
    \alpha=(g_j\gamma^{\Phi(\xi_j)})_{[g_j,a]} \cdot (g_i\gamma^{\Phi(\xi_i)})_{[a,b]} \cdot (g_j\gamma^{\Phi(\xi_j)})_{[b,\infty)}
    \]
    is well-defined. Since $\alpha$ is a geodesic ray in $\Gamma(G,X\sqcup\H)$ from $g_j$ satisfying $Y\-\lim \alpha= Y\-\lim g_j\gamma^{\Phi(\xi_j)}=g_j \xi_j$, the path $g_j^{-1}\alpha$ is a geodesic ray from 1 with $Y\-\lim g_j^{-1}\alpha= \xi_j$. Hence, we have $\mathbf{Lab}(\alpha)=\mathbf{Lab}(g_j^{-1}\alpha)\in \G(\xi_j)$, where $\mathbf{Lab}$ denotes labels of paths in $\Gamma(G,X\sqcup\H)$. On the other hand, we have $\mathbf{Lab}(\alpha)<_{{\rm lex}} \mathbf{Lab}(g_j\gamma^{\Phi(\xi_j)})=\Phi(\xi_j)$ by $(s_{\xi_i,m_i+1},\cdots,s_{\xi_i,m_i+\ell}) <_{{\rm lex}} (s_{\xi_j,m_j+1},\cdots,s_{\xi_j,m_j+\ell})$. This contradicts that $\Phi(\xi_j)$ is minimal in $(\G(\xi_j), \lelex)$. Hence, we have $(s_{\xi_i,m_i+1},\cdots,s_{\xi_i,m_i+\ell}) \ge_{{\rm lex}} (s_{\xi_j,m_j+1},\cdots,s_{\xi_j,m_j+\ell})$. We can also show the converse inequality from minimality of $\Phi(\xi_i)$ in $(\G(\xi_i), \lelex)$. Thus, we get
    \eqref{eq:sxii coincides sxij}, which contradicts \eqref{eq:tail} by $\ell>k$.
\end{proof}

\begin{proof}[Proof of Theorem \ref{thm:main}]
    By Theorem \ref{thm:AH group} $(AH_4)$, there exist a proper infinite subgroup $H$ and a subset $X$ of $G$ such that $H\inj_h (G,X)$. Let $C>0$ as in Proposition \ref{prop:C} and fix $D>0$ satisfying $D\ge 3C$. We also define the subset $Y$ of $G$ by $Y=\{y \in G \mid S(1,y;D)= \emptyset\}$ as in \eqref{eq:Y}. By Theorem \ref{thm:osin} and Lemma \ref{lem:non-elementary}, the Caylay graph $\Gamma(G,Y\sqcup H)$ is hyperbolic, $|\partial\Gamma(G,Y\sqcup H)|>2$, and the action of $G$ on $\Gamma(G,Y\sqcup H)$ is acylindrical. By Proposition \ref{prop:hyperfinite}, the orbit equivalece relation $E_G$ induced by the action of $G$ on $\partial\Gamma(G,Y\sqcup H)$ is hyperfinite. Thus, $S=Y\cup H$ is a generator of $G$ satisfying the statement of Theorem \ref{thm:main}.
\end{proof}

\section{Application to topologically amenable actions}\label{sec:application to topologically amenable actions}

In this short section, by applying Theorem \ref{thm:main} we will prove that any countable acylindrically hyperbolic group admits a topologically amenable action on a Polish space (see Theorem \ref{thm:topologically amenable action}). We begin with introducing some facts about topologically amenable actions and stabilizers of boundary points for a group acting on a hyperbolic space. For more on topologically amenable actions, readers are referred to \cite{BO08}.

\begin{defn}\label{def:topologically amenable action}
    Suppose that $G$ is a countable group and $X$ is a Polish space. A homeomorphic action
$G\act X$ is called \emph{topologically amenable} if for any finite subset $S$ of $G$, any compact set $K$ of $X$,
and any $\e>0$, there exists a continuous map $p\colon X\to\mathrm{Prob}(G)$ such that
\[
\max_{g\in S}\sup_{x\in K}\|p(gx)-g\cdot p(x)\|_1<\e.
\]
\end{defn}

Theorem \ref{thm:hyperfinite and amenable stabilizer imples topologically amenable} immediately follows from Theorem A.1.1 and Theorem A.3.1 of \cite{FKSV} and connects hyperfiniteness and topological amenability of group actions. Note that in Theorem A.3.1 of \cite{FKSV}, the condition that the Polish space is $\sigma$-compact is used only to show that topological amenability implies Borel amenability.

\begin{thm}\label{thm:hyperfinite and amenable stabilizer imples topologically amenable}
    Let $G \act X$ be a homeomorphic action of a countable group $G$ on a
Polish space $X$. If $E_G^X$ is hyperfinite and for any $x\in X$, its stabilizer $\stab_G(x)=\{g\in G \mid gx=x\}$ is amenable, then $G\act X$ is topologically amenable.
\end{thm}

\begin{proof}
     Hyperfiniteness of $E_G^X$ and amenability of stabilizers imply Borel amenability of the action $G\act X$ by \cite[Theorem A.1.1]{FKSV}. Borel amenability trivially implies measure-amenability by their definitions. Measure-amenability implies topological amenability by \cite[Theorem A.3.1]{FKSV}.
\end{proof}

We will next show that boundary stabilizers of a group acting acylindrically on a hyperbolic space are amenable in Lemma \ref{lem:stabilizer of boundary points}. Lemma \ref{lem:rough geodesic} is auxiliary for Lemma \ref{lem:stabilizer of boundary points}. Both of these lemmas should be well-known to experts, but I will record the sketch of proofs for convenience of readers. Note that the $(1,\delta)$-quasi-geodesic ray $p$ in Lemma \ref{lem:rough geodesic} may not be continuous.

\begin{lem}\label{lem:rough geodesic}
    Suppose that $(S,d_S)$ is a $\delta$-hyperbolic geodesic metric space with $\delta\in\NN$, then for any $o\in S$ and $\xi \in \partial S$, there exists a $(1,\delta)$-quasi-geodesic ray $p$ from $o$ to $\xi$ i.e. $p\colon [0,\infty)\to S$ satisfies $p(0)=o$, $\sup_{s,t\in [0,\infty)}|d_S(p(s),p(t))-|s-t||\le \delta$, and $\lim_{s\to\infty}p(s)=\xi$.
\end{lem}

\begin{proof}
    Let $(x_n)_{n=1}^\infty$ be a sequence of points in $S$ that converges to $\xi$. By taking a subsequence of $(x_n)_{n=1}^\infty$ if necessary, we may assume that for any $n,m\in\NN$ with $n\le m$, $(x_n,x_m)_o^S\ge n$ holds since we have $\lim_{n,m\to\infty}(x_n,x_m)_o^S=\infty$. For each $n\in\NN$, take a geodesic path $p_n$ from $o$ to $x_n$. Note that $(x_n,x_m)_o^S\ge n$ implies $d_S(o,x_n)\ge n$ for any $n\in\NN$. Define the map $p\colon [0,\infty) \to S$ as follows. For each $n\in\NN$, $p$ isometrically maps $[n-1,n)$ to the subpath $q_n$ of $p_n$ satisfying $d_S(o,q_{n-})=n-1$. By $d_S(o,q_{1-})=0$, we have $p(0)=o$. Let $s,t\ge0$ with $s\le t$, then there exist $n,m\in\NN$ such that $s\in [n-1,n)$ and $t\in [m-1,m)$. We have $n\le m$. Let $a\in S$ be the unique point on $p_m$ satisfying $d_S(o,a)=s$. Since $a$ and $p(t)$ are both on the geodesic path $p_m$, we have $d_S(a,p(t))=d_S(o,p(t))-d_S(o,a)=t-s$. By $(x_n,x_m)_o^S\ge n>s$, we have $d_S(p(s),a)\le \delta$. Hence, we have
    \[
    |d_S(p(s),p(t))-|s-t||=|d_S(p(s),p(t))-d_S(a,p(t))|\le d_S(p(s),a) \le \delta
    \]
    for any $s,t\ge0$ with $s\le t$. It's not difficult to show that for any $N\in\NN$, if $n,m\ge N$, then $(x_n,p(m))_o^S\ge N-\delta$. This implies $\lim_{s\to\infty}p(s)=\xi$.
\end{proof}

\begin{lem}\label{lem:stabilizer of boundary points}
    Suppose that $(S,d_S)$ is a $\delta$-hyperbolic geodesic metric space with $\delta\in\NN$ and a group $G$ acts on $S$ isometrically and acylindrically. Then, for any $\xi \in \partial S$, the stabilizer $\stab_G(\xi)$ of $\xi$ is virtually cyclic.
\end{lem}

\begin{proof}
    Let $\xi \in \partial S$ and $H=\stab_G(\xi)$. Since the action $G \act S$ is acylindrical, so is $H\act S$. Hence, by \cite[Theorem 1.1]{Osin16} $H$ satisfies exactly one of the following three conditions: (a) $H$ has bounded orbits,
(b) $H$ is virtually cyclic and contains a loxodromic element, (c) $H$ contains infinitely many independent loxodromic elements. Since $H$ fixes $\xi$, (c) cannot occur. We will show that in case (a), $H$ is finite (hence virtually cyclic). Fix $o\in S$ and define $\e$ by $\e=\sup_{g\in H}d(o,go)<\infty$. By Morse lemma (see \cite[Chapter III.H, Theorem 1.7]{BH}), there exists a constant $K(\delta)>0$ such that for any $(1,\delta)$-quasi-geodesic path $q$ and any geodesic path $q'$ from $q_-$ to $q_+$, the Hausdorff distance between $q$ and $q'$ is at most $K(\delta)$. Define $\e'$ by 
$\e'=\e+4K(\delta)+7\delta$. Since $H\act S$ is acylindrical, there exist $R,M\in\NN$ such that for any $x,y\in S$ with $d_S(x,y)\ge R$,
\begin{equation}\label{eq:acylinrical}
    |\{g\in H \mid d_S(x,gx)\le \e' {\rm ~and~} d_S(y,gy)\le \e'\}|\le M.
\end{equation}
By Lemma \ref{lem:rough geodesic}, there exists a $(1,\delta)$-quasi-geodesic ray $p$ from $o$ to $\xi$. Let $g\in H$, then by $g\xi=\xi$ the path $gp$ is a $(1,\delta)$-quasi-geodesic ray from $go$ to $\xi$. Take a real number $s>0$ satisfying $d_S(o,p(s))>\max\{\e+K(\delta),R\}$, then by Morse lemma and $d_S(o,p(s))> d_S(o,go)+K(\delta)$, there exists $t>0$ such that $d_S(p(s),gp(t))\le 2K(\delta)+2\delta$. This implies
\begin{align*}
d_S(gp(s),gp(t))
&=d_S(p(s),p(t))
\le |s-t|+\delta
\le |d_S(o,p(s))-d_S(go,gp(t))| + 3\delta \\
&\le d_S(o,go)+d_S(p(s),gp(t)) + 3\delta
\le \e+ 2K(\delta)+2\delta+ 3\delta.
\end{align*}
Hence, we have $d_S(o,go)\le \e'$ and $d_S(p(s),gp(s))\le d_S(p(s),gp(t))+d_S(gp(t),gp(s))\le \e'$ for any $g\in H$. By $d_S(o,p(s))>R$ and \eqref{eq:acylinrical}, this implies $|H|\le M$.
\end{proof}

Now, we show topological amenability of the boundary action. We restate Corollary \ref{cor:main} as Theorem \ref{thm:topologically amenable action} here. This is an immediate corollary of the above facts and Theorem \ref{thm:main}.

\begin{thm}\label{thm:topologically amenable action}
For any countable acylindrically hyperbolic group $G$, there exists a generating set $S$ of $G$ such that the corresponding Cayley graph $\Gamma(G,S)$ is hyperbolic, $|\partial\Gamma(G,S)|>2$, the natural action of $G$ on $\Gamma(G,S)$ is acylindrical, and the natural action of $G$ on the Gromov boundary $\partial\Gamma(G,S)$ is topologically amenable.
\end{thm}

\begin{proof}
    Take the generating set $S$ of $G$ in Theorem \ref{thm:main}. Since the action $G \act \Gamma(G,S)$ is acylindrical, the stabilizer $\stab_G(\xi)$ is amenable for any $\xi\in\partial\Gamma(G,S)$ by Lemma \ref{lem:stabilizer of boundary points}. This and hyperfiniteness of the action $G\act\partial\Gamma(G,S)$ imply that $G\act\partial\Gamma(G,S)$ is topologically amenable by Theorem \ref{thm:hyperfinite and amenable stabilizer imples topologically amenable}.
\end{proof}

\section{Appendix (more on path representatives)} \label{sec:appendix}
This section is continuation of Section \ref{sec:Gromov}. $G,X,\{H_\lambda\}_{\lambda\in\Lambda},C,D,Y$ are the same as were defined at the beginning of Section \ref{sec:Gromov}. We will list more results on path representatives of the Gromov boundary $\partial\Gamma(G,Y\sqcup\H)$ for possible future use.

After recording the result that $\partial\Gamma(G,Y\sqcup\H)$ is homeomorphic to a certain subset of $\partial\Gamma(G,X\sqcup\H)$ (see Proposition \ref{prop:homeo}), which was essentially proved in Section \ref{sec:Gromov}, we will first show that any two distinct points of $\partial\Gamma(G,Y\sqcup\H)$ can be connected by a bi-infinite geodesic path in $\Gamma(G,X\sqcup\H)$ (see Proposition \ref{prop:existence for two boundary points}). By using this path representative, we will extend the notion of Hull-Osin's separating cosets to pairs of points in $\partial\Gamma(G,Y\sqcup\H)$ (see Definition \ref{def:sep coset for two boundary points}). Finally, we will verify that this generalization of separating cosets to boundary points satisfies similar properties to those of Hull-Osin's separating cosets (see Lemma \ref{lem:penetrated by bi-infinite geodesic} and Proposition \ref{prop:F4}).

In Definition \ref{def:A homeomorphic to Gromov boundary}, we define the set of limit points in $\partial\Gamma(G,X\sqcup\H)$ of all nice geodesic rays in $\Gamma(G,X\sqcup\H)$. This set turns out to be homeomorphic to $\partial\Gamma(G,Y\sqcup\H)$.

\begin{defn}\label{def:A homeomorphic to Gromov boundary}
    We define the subset $A$ of $\partial\Gamma(G,X\sqcup\H)$ by
    \[
    A=\{\xi' \in \partial\Gamma(G,X\sqcup\H) \mid \exists \gamma\colon {\rm geodesic~ray~in~} \Gamma(G,X\sqcup\H) {\rm ~s.t.~}|S(\gamma;D)|=\infty \wedge X\-\lim \gamma=\xi' \}.
    \]
\end{defn}

For $\xi'\in A$, take a geodesic ray $\gamma$ in $\Gamma(G,X\sqcup\H)$ such that $|S(\gamma;D)|=\infty$ and $ X\-\lim \gamma=\xi'$. The geodesic ray $\gamma$ converges to infinity in $\Gamma(G,Y\sqcup\H)$ by Lemma \ref{lem:seq to infinity} and the limit point $Y\-\lim \gamma$ is independent of $\gamma$ taken for $\xi'$ by Proposition \ref{prop:cts from X to Y} and Lemma \ref{lem:well-known}. Hence, this defines the well-defined map $\Psi \colon A \to \partial\Gamma(G,Y\sqcup\H)$ by
\[
\Psi(\xi')=Y\-\lim \gamma.
\]
\begin{prop}\label{prop:homeo}
    The map $\Psi$ is a homeomorphism from $A$ to $\partial\Gamma(G,Y\sqcup\H)$.
\end{prop}

\begin{proof}
    Injectivity and surjectivity of $\Psi$ follow from Corollary \ref{cor:for hyperfinite} and Proposition \ref{prop:existence} respectively. Continuity of $\Psi$ and $\Psi^{-1}$ follow from Proposition \ref{prop:cts from X to Y} and Proposition \ref{prop:for cts} respectively.
\end{proof}

\begin{rem}
    By Proposition \ref{prop:homeo} and the Luzin-Souslin Theorem (see \cite[Corollary 15.2]{Kec95}) for $\Psi^{-1}$, the set $A$ is Borel in $\partial\Gamma(G,X\sqcup\H)$. It's interesting to know whether the geodesic boundary is Borel or not in $\partial\Gamma(G,X\sqcup\H)$. Recall that the geodesic boundary of a geodesic hyperbolic metric space $S$ is the set of all points in the Gromov boundary $\partial S$ that can be realized as the limit point of a geodesic ray in $S$.
\end{rem}

We will now begin our discussion to extend the notion of Hull-Osin's separating cosets to pairs of boundary points, which completes in Definition \ref{def:sep coset for two boundary points}. Definition \ref{def:endpoints of bi-infinite geodesic} sets up notations for the endpoints of a bi-infinite geodesic path.

\begin{defn}\label{def:endpoints of bi-infinite geodesic}
    Suppose that $\gamma=(\cdots,x_{-1},x_0,x_1,\cdots)$ is a bi-infinite geodesic path in $\Gamma(G,X\sqcup\H)$ such that the sequences $(x_n)_{n=1}^\infty$ and $(x_n)_{n=-1}^{-\infty}$ converge to infinity in $\Gamma(G,Y\sqcup\H)$. We denote the limit point $Y\-\lim_{n\to\infty}x_n$ in $\partial\Gamma(G,Y\sqcup\H)$ by $Y\-\lim \gamma_+$ and the limit point $Y\-\lim_{n\to-\infty}x_n$ in $\partial\Gamma(G,Y\sqcup\H)$ by $Y\-\lim \gamma_-$.
\end{defn}

As we did in Proposition \ref{prop:existence}, we first show that two distinct boundary points can be connected by a bi-infinite geodesic path of the smaller Cayley graph.

\begin{prop}\label{prop:existence for two boundary points}
    For any two distinct points $\xi,\eta \in \partial\Gamma(G,Y\sqcup\H)$, there exists a bi-infinite geodesic path $\gamma$ in $\Gamma(G,X\sqcup\H)$ such that $Y\-\lim\gamma_- =\xi$ and $Y\-\lim\gamma_+ =\eta$.
\end{prop}

\begin{proof}
    By Proposition \ref{prop:existence}, fix geodesic rays $\alpha=(1,x_1,x_2,\cdots), \beta=(1,y_1,y_2,\cdots)$ in $\Gamma(G,X\sqcup\H)$ from $1$ such that $Y\-\lim \alpha = \xi$ and $Y\-\lim \beta = \eta$. Since $\xi \neq \eta$, it's straightforward to see that there exist $R,m_0\in\NN$ such that for any $i,j \ge m_0$,
    \begin{align}\label{eq:distinct}
        (x_i,y_j)_1^{Y\cup\H} \le R.
    \end{align}
    For each $n \ge m_0$, fix a geodesic path $\gamma_n$ in $\Gamma(G,X\sqcup\H)$ from $x_n$ to $y_n$. Also, take a geodesic path $[x_n,y_n]$ in $\Gamma(G,Y\sqcup\H)$ for each $n \ge m_0$, then there exists a vertex $z'_n \in [x_n,y_n]$ such that $d_{Y\cup\H}(1,z'_n) \le R+\delta_Y$ by \eqref{eq:distinct} and we can take a vertex $z_n\in \gamma_n$ satisfying $d_{Y\cup\H}(z'_n,z_n) \le M_X$ by Lemma \ref{lem:Haus} (b). This $z_n$ satisfies
    \begin{equation}\label{eq:zn}
        d_{Y\cup\H}(1,z_n) \le R+\delta_Y+M_X
    \end{equation}
    for each $n\ge m_0$. Let $S(\alpha;D)=\{C_1^\alpha\preceq C_2^\alpha \preceq\cdots\}$ and $S(\beta;D)=\{C_1^\beta \preceq C_2^\beta \preceq \cdots\}$ and let each $C_i^\alpha$ (resp. $C_i^\beta$) be a coset of $H_{\lambda_i^\alpha}$ (resp. $H_{\lambda_i^\beta}$). Define $I$ by $I=3(R+\delta_Y+M_X)+2$. We claim that for any $i,n\in\NN$ satisfying $I\le i$ and $d_{X\cup\H}(1,\alpha_{out}(C_{i+1}^\alpha)) \le n$, $\gamma_n$ penetrates $C_i^\alpha$. Indeed, let $S(x_n,1;D)=\{B_1\preceq\cdots\preceq B_m\}$. By applying Lemma \ref{lem: sep coset ordered} to $\alpha$ and $S(1,x_n;D)$, we have $B_{m-(j-1)}=C_j^\alpha$ for any $j\in\{1,\cdots,i\}$. Since we have $3d_{Y\cup\H}(1,z_n)+1 \le i-1$, the path $\gamma_n$ penetrates $C_i^\alpha(=B_{m-(i-1)})$ by applying Lemma \ref{lem:penetration in triangle} to $x_n,1,\gamma_{[x_n,z_n]}$. By applying Lemma \ref{lem:bounded by 3C} to $(\alpha_{[1,x_n]})^{-1}$ and $\gamma_n$, we have
    \begin{equation*}
        \hd_{\lambda_i^\alpha}(\alpha_{out}(C_i^\alpha), (\gamma^{-1}_n)_{out}(C_i^\alpha))
        = \hd_{\lambda_i^\alpha}((\alpha_{[1,x_n]})^{-1}_{in}(C_i^\alpha), (\gamma_n)_{in}(C_i^\alpha))
        \le 3C.
    \end{equation*}
    In the same way, we can also see that for any $i,n\in\NN$ satisfying $I\le i$ and $d_{X\cup\H}(1,\beta_{out}(C_{i+1}^\beta)) \le n$, the path $\gamma_n$ penetrates $C_i^\beta$ and we have $\hd_{\lambda_i^\beta}(\beta_{out}(C_i^\beta), \gamma_{n~out}(C_i^\beta)) \le 3C$.
    Since $\hd_{\lambda_i^\alpha}$ and $\hd_{\lambda_i^\beta}$ are locally finite for any $i\in\NN$, the sets $A_i,B_i$ defined by
    \[
    A_i=\{h\in C_i^\alpha \mid \hd_{\lambda_i^\alpha}(h,\alpha_{out}(C_i^\alpha))\le 3C \}
    {\rm ~~and~~}
    B_i=\{h\in C_i^\beta \mid \hd_{\lambda_i^\beta}(h,\beta_{out}(C_i^\beta))\le 3C \}
    \]
    are finite for any $i\in\NN$. Hence, by the above claim for $i=I$, there exist a subsequence $(\gamma_{1k})_{k=1}^\infty$ of $(\gamma_n)_{n=m_1}^\infty$ and vertices $a_1\in A_I$, $b_1\in B_I$ such that $\{a_1,b_1\}\subset \gamma_{1k}$ for any $k\in\NN$. By repeating this argument for $i=I+1,I+2,\cdots$ and taking subsequences, we can see that there exist a sequence of subsequences $(\gamma_{1k})_{k=1}^\infty \supset (\gamma_{2k})_{k=1}^\infty \supset \cdots$ and vertices $a_n \in A_{I+n-1}$, $b_n \in B_{I+n-1}$ for each $n\in\NN$ such that $\{a_n,\cdots,a_1,b_1,\cdots,b_n\} \subset \gamma_{nk}$ for any $n,k\in\NN$. Take the diagonal sequence $(\gamma_{kk})_{k=1}^\infty$, then for any $n,k\in\NN$ satisfying $k\ge n$, we have
    \begin{align}\label{eq:gamma nn}
        \{a_n,\cdots,a_1,b_1,\cdots,b_n\} \subset \gamma_{kk}.
    \end{align}
    Define the bi-infinite path $\gamma$ in $\Gamma(G,X\sqcup\H)$ by
    \[
    \gamma =\bigcup_{i=2}^\infty \gamma_{nn[a_n,a_{n-1}]} \cup \gamma_{11[a_1,b_1]} \cup \bigcup_{n=2}^\infty \gamma_{nn[b_{n-1},b_n]}.
    \]
    By \eqref{eq:gamma nn}, $\gamma$ is geodesic in $\Gamma(G,X\sqcup\H)$. By $d_{Y\cup\H}(a_n, \alpha_{out}(C_{I+n-1}^\alpha))\le 1, \forall n\in\NN$ and $Y\-\lim_{n\to\infty} \alpha_{out}(C_{I+n-1}^\alpha)=\xi$, we have $Y\-\lim_{n\to\infty} a_n=\xi$. This implies $Y\-\lim\gamma_-=Y\-\lim_{n\to\infty} a_n=\xi$ by Lemma \ref{lem:seq to infinity}. Similarly, $Y\-\lim_{n\to\infty} b_n=\eta$ implies $Y\-\lim\gamma_+=\eta$.
\end{proof}

    \begin{figure}[htbp]
    \begin{center}
    \hspace{0mm} 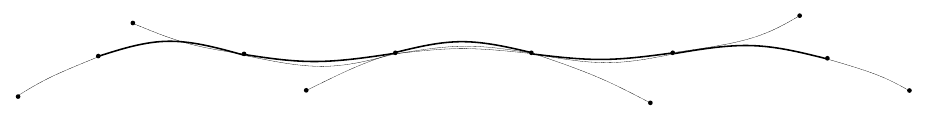
    \end{center}
    \vspace{-3mm}
    \caption{The bi-infinite path $\gamma$ in the proof of Proposition \ref{prop:existence for two boundary points}}
    \label{Prop6_5}
    \end{figure}

As in Definition \ref{def:S(gamma;D)} and Definition \ref{def:order on S gamma}, we next define separating cosets for a bi-infinite geodesic path and align these separating cosets based on the order of their penetration.

\begin{defn}\label{def:sep coset for bi-infinite path}
For a bi-infinite geodesic path $\gamma=(\cdots,x_{-1},x_0,x_1,\cdots)$ in $\Gamma(G,X\sqcup\H)$, we define the set $S(\gamma;D)$ of cosets by
\[
S(\gamma;D)=\bigcup_{n,m\in\ZZ, n<m}S(x_n,x_m;D).
\]
We call an element of $S(\gamma;D)$ a $(\gamma;D)$-\emph{separating coset}.
\end{defn}

\begin{rem}\label{rem:bi-infinite}
In  Definition \ref{def:sep coset for bi-infinite path}, we have $S(\gamma;D)=\bigcup_{n\in\NN}S(x_{-n},x_n;D)$ by Lemma \ref{lem:inclusion}. Also, since $\gamma$ penetrates all cosets in $S(\gamma;D)$ by Lemma \ref{lem:lem4.5 osin}, we can define the relation $\preceq$ on $S(\gamma;D)$ as follows: for any $C_1,C_2 \in S(\gamma;D)$,
\[
C_1\preceq C_2 \iff \exists N\in\ZZ {\rm ~s.t.~} \forall n\le N, d_{X\cup\H}(x_n,\gamma_{in}(C_1))\le d_{X\cup\H}(x_n,\gamma_{in}(C_1)).
\]
We can see that the relation $\preceq$ is a linear order. We write $S(\gamma;D)=\{\cdots \preceq C_{-1} \preceq C_0 \preceq C_1 \preceq\cdots\}$ considering this order.
\end{rem}

As in Corollary \ref{cor:separating coset of ray}, we finally show that two bi-infinite geodesic paths with the same endpoints have the same separating cosets. This enables us to define separating cosets for a pair of boundary points using a bi-infinite geodesic path connecting them.

\begin{lem}\label{lem:well-defined}
    Suppose that $\xi,\eta \in \partial\Gamma(G,Y\sqcup\H)$ are distinct and $\alpha,\beta$ are bi-infinite geodesic paths in $\Gamma(G,X\sqcup\H)$ such that $Y\-\lim \alpha_-=Y\-\lim \beta_-=\xi$ and $Y\-\lim \alpha_+=Y\-\lim \beta_+=\eta$. Then, we have $S(\alpha;D)=S(\beta;D)$.
\end{lem}

\begin{proof}
    Let $\alpha=(\cdots,x_{-1},x_0,x_1,\cdots)$ and $S(\alpha;D)=\{ \cdots \preceq C_{-1} \preceq C_0 \preceq C_1 \preceq\cdots \}$. By $Y\-\lim \alpha_-=Y\-\lim \beta_-$, $Y\-\lim \alpha_+=Y\-\lim \beta_+$, and Corollary \ref{cor:for hyperfinite}, there exists $N\in\NN$ such that $\beta$ penetrates $C_{-n}$ and $C_n$ for any $n\ge N$. For any $i\in\ZZ$, there exists $m \in\NN$ such that $C_i \in S(x_{-m},x_m;D)$ by Remark \ref{rem:bi-infinite}. For $N$ and $m$, there exists $n\ge N$ such that $\alpha_{[x_{-m},x_m]}\subset \alpha_{[\alpha_{out}(C_{-n}),\alpha_{in}(C_n)]}$. Since this implies $C_i \in S(\alpha_{out}(C_{-n}),\alpha_{in}(C_n);D)$ by Lemma \ref{lem:inclusion}, there exists a geodesic path $p$ in $\Gamma(G,X\sqcup\H)$ from $\alpha_{out}(C_{-n})$ to $\alpha_{in}(C_n)$ that essentially penetrates $C_i$. Let $C_{-n},C_n$ be cosets of $H_{\lambda_{-n}}, H_{\lambda_n}$ respectively and let $e_1,e_2$ be the edges of $\Gamma(G,X\sqcup\H)$ with their labels in $H_{\lambda_{-n}}, H_{\lambda_n}$ respectively such that $e_1$ is from $\beta_{in}(C_{-n})$ to $\alpha_{out}(C_{-n})$ and $e_2$ is from $\alpha_{in}(C_{n})$ to $\beta_{out}(C_{n})$. Since we have $d_{X\cup\H}(\alpha_{out}(C_{-n}),\alpha_{in}(C_{n}))=d_{X\cup\H}(\beta_{out}(C_{-n}),\beta_{in}(C_{n}))=d_{X\cup\H}(C_{-n},C_n)$ by Lemma \ref{lem:distance of cosets}, the path $e_1pe_2$ from $\beta_{in}(C_{-n})$ to $\beta_{out}(C_{n})$ is geodesic in $\Gamma(G,X\sqcup\H)$ and essentially penetrates $C_i$. This implies $C_i \in S(\beta_{in}(C_{-n}),\beta_{out}(C_{n});D) \subset S(\beta;D)$. Hence, we have $S(\alpha;D) \subset S(\beta;D)$. Similarly, we can also see $S(\beta;D) \subset S(\alpha;D)$.
\end{proof}

We can now extend the notion of Hull-Osin's separating cosets to a pair of boundary points in the same way as Definition \ref{def:sep coset for vertex and boundry points}.

\begin{defn}\label{def:sep coset for two boundary points}
    For $\xi,\eta \in \partial\Gamma(G,Y\sqcup\H)$ with $\xi\neq\eta$, we take a bi-infinite geodesic path $\gamma$ in $\Gamma(G,X\sqcup\H)$ satisfying $Y\-\lim\gamma_-=\xi$ and $Y\-\lim\gamma_+=\eta$, and define the set $S(\xi,\eta;D)$ of cosets by 
    \[
    S(\xi,\eta;D)=S(\gamma;D).
    \]
    We call an element of $S(\xi,\eta;D)$ a $(\xi,\eta;D)$-\emph{separating coset}. For convenience, we also define $S(\xi,\xi;D)$ by $S(\xi,\xi;D)=\emptyset$ for any $\xi \in \partial\Gamma(G,Y\sqcup\H)$.
\end{defn}

\begin{rem}
    Definition \ref{def:sep coset for two boundary points} is well-defined by Proposition \ref{prop:existence for two boundary points} and Lemma \ref{lem:well-defined}. Also, $S(\xi,\eta;D)$ is exactly the set of all cosets that are essentially penetrated by some bi-infinite geodesic path $\gamma$ in $\Gamma(G,X\sqcup\H)$ satisfying $Y\-\lim\gamma_-=\xi$ and $Y\-\lim\gamma_+=\eta$.
\end{rem}

\begin{figure}[htbp]
  \begin{center}
 \hspace{0mm} 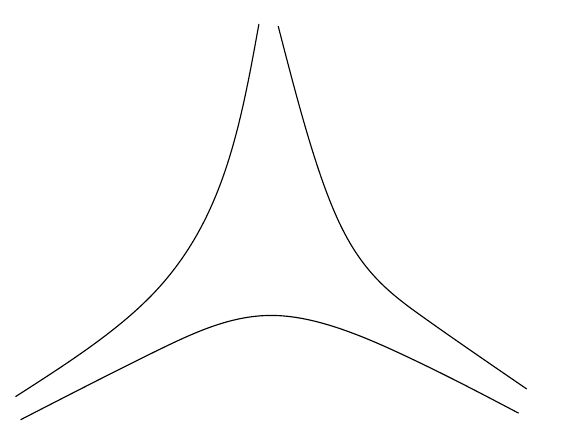
  \end{center}
   \vspace{-3mm}
  \caption{Proof of Lemma \ref{lem:penetrated by bi-infinite geodesic}}
  \label{Lem5_11}
\end{figure}

Our next goal is to show Proposition \ref{prop:F4}, which is an analogue of \cite[Lemma 3.9]{HO13}. We first prepare auxiliary results. For distinct elements $\xi,\eta\in \partial\Gamma(G,Y\sqcup\H)$, if a bi-infinite geodesic path in $\Gamma(G,X\sqcup\H)$ satisfies $Y\-\lim \gamma_-=\xi$ and $Y\-\lim \gamma_+=\eta$, then we say that $\gamma$ is \emph{from} $\xi$ \emph{to} $\eta$.

Lemma \ref{lem:penetrated by bi-infinite geodesic} below is analogous to Lemma \ref{lem:lem4.5 osin}.

\begin{lem}\label{lem:penetrated by bi-infinite geodesic}
Let $D\ge 6C$. For any distinct elements $\xi,\eta,\zeta\in \partial\Gamma(G,Y\sqcup\H)$ and any $B\in S(\xi,\eta;D)$, $B$ is either penetrated by all bi-infinite geodesic paths in $\Gamma(G,X\sqcup\H)$ from $\xi$ to $\zeta$ or penetrated by all bi-infinite geodesic paths in $\Gamma(G,X\sqcup\H)$ from $\zeta$ to $\eta$.
\end{lem}

\begin{proof}
Let $S(\xi,\eta;D)=\{ \cdots \preceq C_{-1} \preceq C_0 \preceq C_1 \preceq\cdots \}$ and suppose that there exist $j\in\ZZ$ and a bi-infinite geodesic path $q$ in $\Gamma(G,X\sqcup\H)$ from $\xi$ to $\zeta$ that doesn't penetrate $C_j$. Take a bi-infinite geodesic path $p$ in $\Gamma(G,X\sqcup\H)$ from $\xi$ to $\eta$ that essentially penetrates $C_j$. For any bi-infinite geodesic path $\alpha$ in $\Gamma(G,X\sqcup\H)$ from $\zeta$ to $\eta$, there exists an $H_\lambda$-coset $B'$ and $i,k\in\ZZ$ with $i<j<k$ such that $B'$ (resp. $C_i$, $C_k$) is penetrated by both $q$ and $\alpha$ (resp. $p$ and $q$, $p$ and $\alpha$) by Corollary \ref{cor:for hyperfinite}. Note $B'\neq C_j$ since $q$ doesn't penetrate $C_j$. Let $C_i,C_j,C_k$ be cosets of $H_{\lambda_i},H_{\lambda_j},H_{\lambda_k}$ respectively and let $e_1,e_2,e_3$ be the edges in $\Gamma(G,X\sqcup\H)$ with their labels in $H_{\lambda_i},H_{\lambda},H_{\lambda_k}$ respectively such that $e_1$ is from $p_{out}(C_i)$ to $q_{out}(C_i)$, $e_2$ is from $q_{in}(B')$ to $\alpha_{out}(B')$, and $e_3$ is from $\alpha_{in}(C_k)$ to $p_{in}(C_k)$. If $\alpha$ doesn't penetrate $C_j$, then the component of $p_{[p_{out}(C_i), p_{in}(C_k)]}$ corresponding to $C_j$ is isolated in the geodesic hexagon $e_1q_{[q_{out}(C_i),q_{in}(B')]} e_2 \alpha_{[\alpha_{out}(B'), \alpha_{in}(C_k)]} e_3 (p_{[p_{out}(C_i),p_{in}(C_k)]})^{-1}$ by $C_j\notin \{C_i,C_k,B'\}$. This implies $\hd_{\lambda_j}(p_{in}(C_j),p_{out}(C_j))\le 6C$ by Proposition \ref{prop:C}. This contradicts that $p$ essentially penetrates $C_j$ since we assume $D\ge 6C$. Thus, $\alpha$ penetrates $C_j$.
\end{proof}

Lemma \ref{lem:appendix penetrate} below means that if a geodesic ray converges to one endpoint of a bi-infinite geodesic path, then the geodesic ray penetrates separating cosets of the bi-infinite path in the same order as the order of the separating cosets.

\begin{figure}[htbp]
  \begin{center}
 \hspace{0mm} 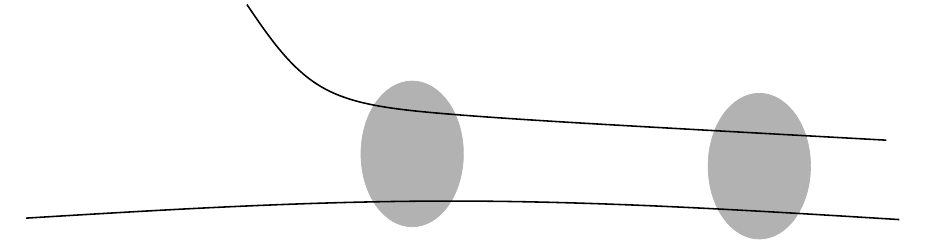
  \end{center}
   \vspace{-3mm}
  \caption{Proof of Lemma \ref{lem:appendix penetrate}}
  \label{Lem5_12}
\end{figure}

\begin{lem}\label{lem:appendix penetrate}
    Let $D\ge 4C$. Suppose that $\xi,\eta \in \partial\Gamma(G,Y\sqcup\H)$ are distinct and $\alpha$ is a geodesic ray in $\Gamma(G,X\sqcup\H)$ from $\alpha_- \in G$ to $\eta$. Let $S(\xi,\eta;D)=\{ \cdots \preceq C_{-1} \preceq C_0 \preceq C_1 \preceq\cdots \}$. If $\alpha$ penetrates $C_i$ for some $i\in \ZZ$, then the subpath $\alpha_{[\alpha_{out}(C_i),\infty)}$ penetrates $C_{i+1}$.
\end{lem}

\begin{proof}
    By $C_{i+1}\in S(\xi,\eta;D)$, there exists a bi-infinite geodesic path $\beta$ in $\Gamma(X\sqcup\H)$ from $\xi$ to $\eta$ that essentially penetrates $C_{i+1}$. By $Y\-\lim \alpha=Y\-\lim \beta_+=\eta$ and Corollary \ref{cor:for hyperfinite}, there exists $j$ with $j>i+1$ such that $C_j$ is penetrated by both $\alpha$ and $\beta$ and satisfies $d_{X\cup\H}(\alpha_-,\alpha_{in}(C_i))<d_{X\cup\H}(\alpha_-,\alpha_{in}(C_j))$. Note that $\beta$ penetrates $C_i$. Let $C_i,C_{i+1},C_j$ be cosets of $H_{\lambda_i},H_{\lambda_{i+1}}, H_{\lambda_j}$ respectively and let $e_1,e_2$ be the edges in $\Gamma(G,X\sqcup\H)$ with their labels in $H_{\lambda_i},H_{\lambda_j}$ respectively such that $e_1$ is from $\alpha_{out}(C_i)$ to $\beta_{out}(C_i)$ and $e_2$ is from $\alpha_{in}(C_j)$ to $\beta_{in}(C_j)$. If the subpath $\alpha_{[\alpha_{out}(C_i),\infty)}$ doesn't penetrate $C_{i+1}$, then the component of $\beta_{[\beta_{out}(C_i),\beta_{in}(C_j)]}$ corresponding to $C_{i+1}$ is isolated in the geodesic quadrilateral $e_1 \beta_{[\beta_{out}(C_i),\beta_{in}(C_j)]} e_2^{-1} (\alpha_{[\alpha_{out}(C_i),\alpha_{in}(C_j)]})^{-1}$ by $C_{i+1}\notin \{C_i,C_j\}$. This implies $\hd_{\lambda_{i+1}}(\beta_{in}(C_{i+1}),\beta_{out}(C_{i+1}))\le 4C$ by Proposition \ref{prop:C}. This contradicts that $\beta$ essentially penetrates $C_{i+1}$ since we assume $D\ge 4C$. Thus, $\alpha_{[\alpha_{out}(C_i),\infty)}$ penetrates $C_{i+1}$.
\end{proof}

Lemma \ref{lem:appendix geodesic} below enables us to create a new bi-infinite geodesic ray by concatenating two geodesic paths.

\begin{figure}[htbp]
  \begin{center}
 \hspace{0mm} 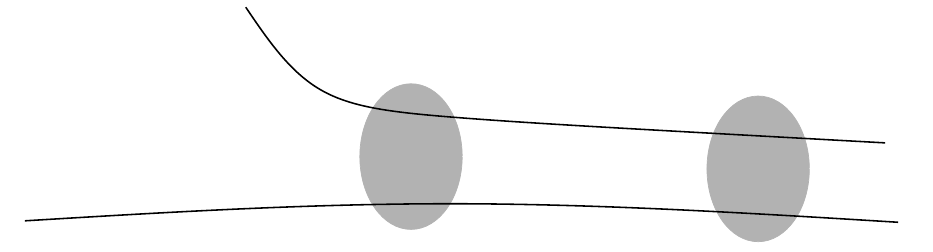
  \end{center}
   \vspace{-3mm}
  \caption{Proof of Lemma \ref{lem:appendix geodesic}}
  \label{Lem5_13}
\end{figure}

\begin{lem}\label{lem:appendix geodesic}
    Suppose that $\xi,\eta \in \partial\Gamma(G,Y\sqcup\H)$ are distinct, $\alpha$ is a geodesic ray in $\Gamma(X\sqcup\H)$ from $\alpha_-\in G$ to $\eta$, and $\beta$ is a bi-infinite geodesic path in $\Gamma(X\sqcup\H)$ from $\xi$ to $\eta$. If $\alpha$ and $\beta$ penetrate an $H_\lambda$-coset $B$ satisfying $\hd_{H_\lambda}(\beta_{in}(B),\beta_{out}(B)) > 3C$ and $e$ is the edge in $\Gamma(G,X\sqcup\H)$ from $\beta_{in}(B)$ to $\alpha_{out}(B)$ whose label is in $H_{\lambda}$, then the bi-infinite path $\beta_{(-\infty,\beta_{in}(B)]} e \alpha_{[\alpha_{out}(B),\infty)}$ is geodesic in $\Gamma(G,X\sqcup\H)$.
\end{lem}

\begin{proof}
    Let $\alpha=(x_0,x_1,\cdots)$ and $\beta=(\cdots,y_{-1},y_0,y_1,\cdots)$, and let $N\in\ZZ$ satisfy $y_N=\beta_{in}(B)$. Fix $i\in\ZZ$ with $i<N$. There exist $k\in\NN$ and a geodesic path $p$ in $\Gamma(G,X\sqcup\H)$ from $y_i$ to $x_k$ such that the path $p\alpha_{[x_k,\infty)}$ is geodesic in $\Gamma(G,X\sqcup\H)$ by Lemma \ref{lem:well-known}. By $Y\-\lim \alpha=Y\-\lim \beta_+=\eta$ and Corollary \ref{cor:for hyperfinite}, there exists a $H_{\lambda_1}$-coset $B_1 \in S(\alpha,D)$ such that $\beta$ penetrates $B_1$ and we have $d_{X\cup\H}(x_0, x_k)<d_{X\cup\H}(x_0,\alpha_{in}(B_1))$ and $d_{X\cup\H}(y_i, \beta_{in}(B))<d_{X\cup\H}(y_i,\beta_{in}(B_1))$. Let $e_1$ be the edge in $\Gamma(G,X\sqcup\H)$ from $\alpha_{in}(B_1)$ to $\beta_{in}(B_1)$ whose label is in $H_{\lambda_1}$. Define the path $q$ by $q=p\alpha_{[x_k,\alpha_{in}(B_1)]}$ and consider the geodesic triangle $\Delta=qe_1 (\beta_{[y_i,\beta_{in}(B_1)]})^{-1}$. The component of $\beta_{[y_i,\beta_{in}(B_1)]}$ corresponding to $B$ cannot be isolated in $\Delta$ by Proposition \ref{prop:C} and $\hd_{H_\lambda}(\beta_{in}(B),\beta_{out}(B)) > 3C$. By this and $B\neq B_1$, the path $q$ penetrates $B$. Hence, we have
    \begin{align*}
        d_{X\cup\H}(y_i,\beta_{in}(B))&=d_{X\cup\H}(y_i,q_{in}(B))~ (=d_{X\cup\H}(y_i,B)) \\
        {\rm and~~~~~}
        d_{X\cup\H}(\alpha_{out}(B), \alpha_{in}(B_1))&=d_{X\cup\H}(q_{out}(B), \alpha_{in}(B_1)) ~(=d_{X\cup\H}(\alpha_{in}(B_1),B))
    \end{align*}
    by Lemma \ref{lem:lem4.6 osin}. This implies
    \begin{align*}
        |\beta_{[y_i,\beta_{in}(B)]} e \alpha_{[\alpha_{out}(B),\alpha_{in}(B_1)]}|
        &\le d_{X\cup\H}(y_i,\beta_{in}(B)) + 1 + d_{X\cup\H}(\alpha_{out}(B), \alpha_{in}(B_1)) \\
        &=d_{X\cup\H}(y_i,q_{in}(B))+1+d_{X\cup\H}(q_{out}(B), \alpha_{in}(B_1))\\
        &=|q|=d_{X\cup\H}(y_i,\alpha_{in}(B_1)).
    \end{align*}
    Hence, the path $r$ defined by $r=\beta_{[y_i,\beta_{in}(B)]} e \alpha_{[\alpha_{out}(B),\alpha_{in}(B_1)]}$ is geodesic in $\Gamma(G,X\sqcup\H)$ and has the same endpoints as $q$. This implies that $\beta_{[y_i,\beta_{in}(B)]} e \alpha_{[\alpha_{out}(B),\infty)}(=r\alpha_{[\alpha_{in}(B_1),\infty)})$ is geodesic in $\Gamma(G,X\sqcup\H)$ since $q\alpha_{[\alpha_{in}(B_1),\infty)}(= p \alpha_{[x_k,\infty)]})$ is geodesic in $\Gamma(G,X\sqcup\H)$. Since $i$ with $i<N$ is arbitrary, this implies that $\beta_{(-\infty,\beta_{in}(B)]} e \alpha_{[\alpha_{out}(B),\infty)}$ is geodesic in $\Gamma(G,X\sqcup\H)$.
\end{proof}

\begin{figure}[htbp]
  \begin{center}
 \hspace{0mm} 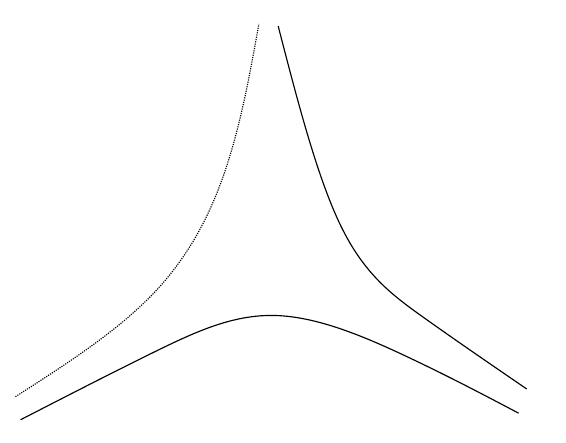
  \end{center}
   \vspace{-3mm}
  \caption{Proof of Proposition \ref{prop:F4}}
  \label{Prop5_14}
\end{figure}

We are now ready to show Proposition \ref{prop:F4}. The proof is similar to \cite[Lemma 3.9]{HO13} modulo the above auxiliary lemmas.

\begin{prop}\label{prop:F4}
    Let $D\ge 11C$. For any $\xi,\eta,\zeta \in G \cup \partial\Gamma(G,Y\sqcup\H)$, $S(\xi,\eta;D)$ can be decomposed into $S(\xi,\eta;D)=S'\sqcup S'' \sqcup F$ such that $S'\subset S(\xi,\zeta;D)$, $S''\subset S(\zeta,\eta;D)$, and $|F|\le 4$.
\end{prop}

\begin{proof}
We will only show the case where $\xi,\eta,\zeta \in \partial\Gamma(G,Y\sqcup\H)$ and $\xi,\eta,\zeta$ are all distinct, because the proof of other cases is similar. Let $S(\xi,\eta;D)=\{ \cdots \preceq C_{-1} \preceq C_0 \preceq C_1 \preceq\cdots \}$ and define $\P_\xi=\{i\in\ZZ \mid \forall\gamma\colon {\rm bi}\-{\rm infinite~geodesic~path~in~}\Gamma(G,X\sqcup\H) {\rm ~from~}\xi {\rm ~to~}\zeta {\rm ~penetrates~} C_i \}$ and $\P_\eta=\{i\in\ZZ \mid \forall\gamma\colon {\rm bi}\-{\rm infinite~geodesic~path~in~}\Gamma(G,X\sqcup\H) {\rm ~from~}\zeta {\rm ~to~}\eta {\rm ~penetrates~} C_i \}$. By Lemma \ref{lem:penetrated by bi-infinite geodesic}, we have $\ZZ=\P_\xi\cup \P_\eta$. Suppose for contradiction that there exists a sequence $(i_k)_{k=1}^\infty$ in $\P_\xi$ such that $\lim_{k\to\infty}i_k=\infty$. Take bi-infinite geodesic paths $p,q$ in $\Gamma(G,X\sqcup\H)$ such that $p$ is from $\xi$ to $\zeta$ and $q$ is from $\xi$ to $\eta$. Since $i_k\in \P_\xi$ implies $d_{Y\cup\H}(p_{in}(C_{i_k}), q_{in}(C_{i_k}))\le 1$ for any $k\in\NN$, we have $\zeta=\lim_{k\to\infty} p_{in}(C_{i_k})=\lim_{k\to\infty} q_{in}(C_{i_k})=\xi$. This contradicts our assumption that $\xi,\eta,\zeta$ are distinct. Hence, there exists $i_1\in \ZZ$ such that $\P_\xi \cap[i_1,\infty)=\emptyset$. Similarly, we can see $\P_\eta\cap (-\infty,i_2]=\emptyset$ for some $i_2\in\ZZ$. In particular, $\P_\eta$ is nonempty and $\min \P_\eta$ exists. Define $N$ by $N=\min \P_\eta$.

We claim $\{C_i \mid i\ge N+2\}\subset S(\zeta,\eta;D)$. Fix a bi-infinite geodesic path $\beta$ in $\Gamma(G,X\sqcup\H)$ from $\zeta$ to $\eta$. By $N\in \P_\eta$, the path $\beta$ penetrates $C_N$. Hence, the subpath $\beta_{[\beta_{out}(C_N),\infty)}$ penetrates $C_{N+1}$ by Lemma \ref{lem:appendix penetrate}. Let $C_{N+1}$ be an $H_\lambda$-coset, then we have $\hd_{\lambda}(\beta_{in}(C_{N+1}), \beta_{out}(C_{N+1}))> 3C$. Indeed, take a bi-infinite geodesic path $p'$ in $\Gamma(G,X\sqcup\H)$ from $\xi$ to $\eta$ that essentially penetrates $C_N$. By applying Lemma \ref{lem:bounded by 4C} to $C_N$ and $C_{N+1}$, we have $\hd_{\lambda}(p'_{in}(C_{N+1}),\beta_{in}(C_{N+1})) \le 4C$. By the same argument for $C_{N+1}$ and $C_{N+2}$, we can also see $\hd_{\lambda}(p'_{out}(C_{N+1}),\beta_{out}(C_{N+1}))\le 4C$. Since $p'$ essentially penetrates $C_{N+1}$, this implies
\begin{equation}
\begin{split}
    &\hd_{\lambda}(\beta_{in}(C_{N+1}), \beta_{out}(C_{N+1})) \\
    &\ge \hd_{\lambda}(p'_{in}(C_{N+1}),p'_{out}(C_{N+1})) - \hd_{\lambda}(\beta_{in}(C_{N+1}),p'_{in}(C_{N+1}))-\hd_{\lambda}(\beta_{out}(C_{N+1}),p'_{out}(C_{N+1}))\\
    &> D-4C-4C \ge 3C.
\end{split}
\end{equation}
For any $i\ge N+2$, there exists a bi-infinite geodesic path $\alpha$ in $\Gamma(G,X\sqcup\H)$ from $\xi$ to $\eta$ that essentially penetrates $C_i$. Note that $\alpha$ penetrates $C_{N+1}$. Let $e$ be the edge in $\Gamma(G,X\sqcup\H)$ from $\beta_{in}(C_{N+1})$ to $\alpha_{out}(C_{N+1})$ whose label is in $H_{\lambda}$. By Lemma \ref{lem:appendix geodesic}, the bi-infinite path $\gamma$ defined by $\gamma=\beta_{(-\infty,\beta_{in}(C_{N+1})]} e \alpha_{[\alpha_{out}(C_{N+1}),\infty)}$ is geodesic in $\Gamma(G,X\sqcup\H)$ and essentially penetrates $C_i$. This implies $C_i\in S(\gamma;D)$. On the other hand, we have $S(\gamma;D)=S(\zeta,\eta;D)$ by $Y\-\lim \gamma_-=Y\-\lim \beta_-=\zeta$ and $Y\-\lim \gamma_+=Y\-\lim \alpha_+=\eta$. Thus, we have $\{C_i \mid i\ge N+2\}\subset S(\zeta,\eta;D)$. Similarly, we can also see $\{C_i \mid i\le N-3\}\subset S(\xi,\zeta;D)$ since we have $N-1\in \P_\xi$ by $\ZZ=\P_\xi\cup \P_\eta$ and $N=\min\P_\eta$. Thus, we get the desired decomposition by defining $S',S'',F$ by $S'=\{C_i \mid i\le N-3\}$, $S''=\{C_i \mid i\ge N+2\}$, and $F=\{C_{N-2},C_{N-1},C_{N},C_{N+1}\}$.
\end{proof}




\begin{thebibliography}{10}

\bibitem{Bow}
Brian~H. Bowditch, \emph{Uniform hyperbolicity of the curve graphs}, Pacific J. Math. \textbf{269} (2014), no.~2, 269--280. \MR{3238474}

\bibitem{BH}
Martin~R. Bridson and Andr\'{e} Haefliger, \emph{Metric spaces of non-positive curvature}, Grundlehren der mathematischen Wissenschaften [Fundamental Principles of Mathematical Sciences], vol. 319, Springer-Verlag, Berlin, 1999. \MR{1744486}

\bibitem{BO08}
Nathanial~P. Brown and Narutaka Ozawa, \emph{{$C^*$}-algebras and finite-dimensional approximations}, Graduate Studies in Mathematics, vol.~88, American Mathematical Society, Providence, RI, 2008. \MR{2391387}

\bibitem{CJM+}
Clinton Conley, Steve Jackson, Andrew Marks, Brandon Seward, and Robin Tucker-Drob, \emph{Borel asymptotic dimension and hyperfinite equivalence relations}, 2022, Preprint, arXiv:2009.06721.

\bibitem{DGO}
F.~Dahmani, V.~Guirardel, and D.~Osin, \emph{Hyperbolically embedded subgroups and rotating families in groups acting on hyperbolic spaces}, Mem. Amer. Math. Soc. \textbf{245} (2017), no.~1156, v+152. \MR{3589159}

\bibitem{DJK94}
R.~Dougherty, S.~Jackson, and A.~S. Kechris, \emph{The structure of hyperfinite {B}orel equivalence relations}, Trans. Amer. Math. Soc. \textbf{341} (1994), no.~1, 193--225. \MR{1149121}

\bibitem{FM}
Jacob Feldman and Calvin~C. Moore, \emph{Ergodic equivalence relations, cohomology, and von {N}eumann algebras. {I}}, Trans. Amer. Math. Soc. \textbf{234} (1977), no.~2, 289--324. \MR{578656}

\bibitem{FKSV}
Joshua Frisch, Alexander Kechris, Forte Shinko, and Zoltán Vidnyánszky, \emph{Realizations of countable borel equivalence relations}, 2023, Preprint, arXiv:2109.12486.

\bibitem{GJ}
Su~Gao and Steve Jackson, \emph{Countable abelian group actions and hyperfinite equivalence relations}, Invent. Math. \textbf{201} (2015), no.~1, 309--383. \MR{3359054}

\bibitem{GS}
Dominik Gruber and Alessandro Sisto, \emph{Infinitely presented graphical small cancellation groups are acylindrically hyperbolic}, Ann. Inst. Fourier (Grenoble) \textbf{68} (2018), no.~6, 2501--2552. \MR{3897973}

\bibitem{HSS}
Jingyin Huang, Marcin Sabok, and Forte Shinko, \emph{Hyperfiniteness of boundary actions of cubulated hyperbolic groups}, Ergodic Theory Dynam. Systems \textbf{40} (2020), no.~9, 2453--2466. \MR{4130811}

\bibitem{HO13}
Michael Hull and Denis Osin, \emph{Induced quasicocycles on groups with hyperbolically embedded subgroups}, Algebr. Geom. Topol. \textbf{13} (2013), no.~5, 2635--2665. \MR{3116299}

\bibitem{JKL02}
S.~Jackson, A.~S. Kechris, and A.~Louveau, \emph{Countable {B}orel equivalence relations}, J. Math. Log. \textbf{2} (2002), no.~1, 1--80. \MR{1900547}

\bibitem{Kar}
Chris Karpinski, \emph{Hyperfiniteness of boundary actions of relatively hyperbolic groups}, 2022, Preprint, arXiv:2212.00236.

\bibitem{Kec95}
Alexander~S. Kechris, \emph{Classical descriptive set theory}, Graduate Texts in Mathematics, vol. 156, Springer-Verlag, New York, 1995. \MR{1321597}

\bibitem{MS}
Timoth\'{e}e Marquis and Marcin Sabok, \emph{Hyperfiniteness of boundary actions of hyperbolic groups}, Math. Ann. \textbf{377} (2020), no.~3-4, 1129--1153. \MR{4126891}

\bibitem{MO}
Ashot Minasyan and Denis Osin, \emph{Acylindrical hyperbolicity of groups acting on trees}, Math. Ann. \textbf{362} (2015), no.~3-4, 1055--1105. \MR{3368093}

\bibitem{NV}
Petr Naryshkin and Andrea Vaccaro, \emph{Hyperfiniteness and borel asymptotic dimension of boundary actions of hyperbolic groups}, 2023, Preprint, arXiv:2306.02056.

\bibitem{Osin16}
D.~Osin, \emph{Acylindrically hyperbolic groups}, Trans. Amer. Math. Soc. \textbf{368} (2016), no.~2, 851--888. \MR{3430352}

\bibitem{PS}
Piotr Przytycki and Marcin Sabok, \emph{Unicorn paths and hyperfiniteness for the mapping class group}, Forum Math. Sigma \textbf{9} (2021), Paper No. e36, 10. \MR{4252215}

\bibitem{Weiss}
Benjamin Weiss, \emph{Measurable dynamics}, Conference in modern analysis and probability ({N}ew {H}aven, {C}onn., 1982), Contemp. Math., vol.~26, Amer. Math. Soc., Providence, RI, 1984, pp.~395--421. \MR{737417}

\end{thebibliography}

\providecommand{\bysame}{\leavevmode\hbox to3em{\hrulefill}\thinspace}
\providecommand{\MR}{\relax\ifhmode\unskip\space\fi MR }
\providecommand{\MRhref}[2]{%
  \href{http://www.ams.org/mathscinet-getitem?mr=#1}{#2}
}
\providecommand{\href}[2]{#2}

\vspace{5mm}

\noindent  Department of Mathematics, Vanderbilt University, Nashville 37240, U.S.A.

\noindent E-mail: \emph{koichi.oyakawa@vanderbilt.edu}

\end{document}